\documentclass[final,onefignum,onetabnum]{siamonline190516}

\usepackage{amsfonts,amsopn}
\usepackage{multirow,adjustbox}
\usepackage{subcaption}
\usepackage{algorithm,algorithmic}
\ifpdf
  \DeclareGraphicsExtensions{.eps,.pdf,.png,.jpg}
\else
  \DeclareGraphicsExtensions{.eps}
\fi

\usepackage{enumitem}
\usepackage{physics}
\setlist[enumerate]{leftmargin=.5in}
\setlist[itemize]{leftmargin=.5in}

\newsiamremark{remark}{Remark}
\newsiamremark{example}{Example}
\crefname{example}{Example}{Examples}

\headers{DMD for Multiplicative Noise}{M.~Lee and J.~Park}

\title{An Optimized Dynamic Mode Decomposition Model Robust to Multiplicative Noise\thanks{Submitted to the editors DATE.
\funding{Minwoo Lee was supported by the National Research Foundation of Korea~(NRF) grant funded by the Korean government~(MSIT)~(No.~2021R1G1A1091278). Jongho Park was supported by NRF grant funded by MSIT~(No.~2021R1C1C2095193).
}}}

\author{Minwoo Lee\thanks{Department of Mechanical Engineering, Hanbat National University, Daejeon, 34158, Korea 
  (\email{mwlee@hanbat.ac.kr}).}
\and Jongho Park\thanks{Natural Science Research Institute, KAIST, Daejeon, 34141, Korea
  (\email{jongho.park@kaist.ac.kr}, \url{https://sites.google.com/view/jonghopark}).}
}

\ifpdf
\hypersetup{
  pdftitle={An Optimized Dynamic Mode Decomposition Model Robust to Multiplicative Noise},
  pdfauthor={M.~Lee and J.~Park}
}
\fi

\newcommand\gap{\hspace{0.1cm}}
\newcommand\smallgap{\hspace{0.05cm}}

\newcommand\rT{\mathrm{T}}
\newcommand\tH{\widetilde{H}}
\newcommand\cD{\mathcal{D}}
\newcommand\cE{\mathcal{E}}
\newcommand\cJ{\mathcal{J}}
\newcommand\cA{\mathcal{A}}
\newcommand\cN{\mathcal{N}}
\newcommand\cS{\mathcal{S}}
\newcommand\cV{\mathcal{V}}
\newcommand\tcE{\widetilde{\mathcal{E}}}
\newcommand\tcS{\widetilde{\mathcal{S}}}

\newcommand\intO{\int_{\Omega}}

\newcommand{\exact}{\mathrm{exact}}
\newcommand{\recon}{\mathrm{recon}}
\newcommand{\clean}{\mathrm{clean}}
\newcommand{\Ma}{\mathrm{Ma}}

\let\Re\relax
\DeclareMathOperator{\Re}{Re}
\let\Im\relax
\DeclareMathOperator{\Im}{Im}
\DeclareMathOperator*{\argmin}{\arg\min}
\DeclareMathOperator*{\argmax}{\arg\max}
\DeclareMathOperator{\proj}{proj}

\begin{document}

\maketitle

\begin{abstract}
Dynamic mode decomposition~(DMD) is an efficient tool for decomposing spatio-temporal data into a set of low-dimensional modes, yielding the oscillation frequencies and the growth rates of physically significant modes. In this paper, we propose a novel DMD model that can be used for dynamical systems affected by multiplicative noise. We first derive a maximum a posteriori~(MAP) estimator for the data-based model decomposition of a linear dynamical system corrupted by certain multiplicative noise. Applying penalty relaxation to the MAP estimator, we obtain the proposed DMD model whose epigraphical limits are the MAP estimator and the conventional optimized DMD model. We also propose an efficient alternating gradient descent method for solving the proposed DMD model, and analyze its convergence behavior. The proposed model is demonstrated on both the synthetic data and the numerically generated one-dimensional combustor data, and is shown to have superior reconstruction properties compared to state-of-the-art DMD models. Considering that multiplicative noise is ubiquitous in numerous dynamical systems, the proposed DMD model opens up new possibilities for accurate data-based modal decomposition. 
\end{abstract}



\section{Introduction} 
Various natural and engineered systems exhibit complex spatio-temporal behavior, which often involves nonlinear and high-dimensional dynamics. In many cases, however, the system's dynamics are governed by a few significant modes that represent the coherent features of the system. From a practical point of view, it is essential to extract these modes from the experimental data, so as to identify the fundamental dynamics, analyze the underlying physics, and build a low-dimensional model of the system~\cite{CTR:2012,holmes2012}. Over the past few decades, various data-based modal decomposition techniques have been proposed and applied to analyze complex dynamical systems, including fluid flow~\cite{Schmid:2010}, combustion system~\cite{luong2021}, neural activity recording~\cite{brunton2016}, and spread of infectious disease~\cite{proctor2015}, among many others.

Two classes of data-based modal decomposition techniques are commonly used: proper orthogonal decomposition~(POD) and dynamic mode decomposition~(DMD). POD, which is also known as the principal component analysis in the statistics community, is a method of obtaining a low-dimensional approximation by projecting the full dynamical system onto a set of spatially orthogonal basis functions~\cite{rowley2005,sirovich1987}. Although POD can efficiently decompose a physical field to a lower order system, it suffers from several limitations. For instance, the basis functions drawn from POD do not necessarily represent the physically significant modes. Thus, a reduced-order model constructed from POD can be inaccurate due to the user's wrong choice of modes~\cite{ilak2008}. Furthermore, POD is sensitive to the data used and is therefore difficult to be used when the experimental data is contaminated~\cite{rathinam2003}. 

On the contrary, the second approach, DMD, computes eigenvalues and eigenvectors of a linear-approximated model that represent the full dynamics of the system~\cite{Schmid:2010, TRLBK:2014}. Specifically, a dynamical system $\dot{z}=f(z(t))$ at equally spaced time space  $\{t_n \}_{1 \leq n \leq N}$ is approximated as $\xi^{n+1}=A \xi^n$, where $\xi^n \in \mathbb{C}^M$, $1 \leq n\leq N$, is a snapshot at time $t_n$ and $A$ is a linear operator. Then, a set of eigenvalues and eigenvectors of $A$ is found from the snapshots $\{ \xi^n \}_n$. When applied to nonlinear systems, DMD can be viewed as a method for finding approximate modes of the Koopman operator~\cite{mezic2005}. Unlike POD, growth rates and frequencies associated with each mode can be drawn from DMD, enabling the construction of a physically meaningful low-order model~\cite{CTR:2012}. However, owing to the fact that DMD uses pairs of data~(one snapshot and the next), rather than the whole set of data at once, DMD is prone to the bias caused by sensor noise~\cite{DHWR:2016, hemati2017}. 

Addressing this issue, the \textit{optimized} DMD was proposed in~\cite{CTR:2012}; it processes the whole snapshot data at once. Specifically, this algorithm finds a set of eigenvalues that minimizes the residual between the original spatio-temporal data and the reduced-order model. Although the optimized DMD uses a computationally expensive optimization method, namely the Nelder--Mead simplex method, it is shown that the optimized DMD can reduce the bias of the original DMD. Later, Askham and Kutz~\cite{AK:2018} improved the optimized DMD by reshaping the above-mentioned minimization problem and adopting the Levenberg--Marquardt algorithm~\cite{AAAM:2011,Marquardt:1963}. They showed that their proposed algorithm is robust to noise and does not require the original data to be evenly spaced in time. Recently, Askham et al.~\cite{Askham:2022} further advanced the optimized DMD method by incorporating robust statistics. In particular, recognizing that optimized DMD is sensitive to outliers, the authors applied robust penalties and parameter constraints for bias reduction and future state prediction.

In this study, we build on the established optimized DMD algorithms by tailoring the presentation of Askham and Kutz~\cite{AK:2018}. Specifically, we focus on the fact that the existing optimized DMD algorithms such as~\cite{AK:2018} are designed particularly to handle additive Gaussian noise~(this claim will be discussed in \cref{Sec:additive} with details). In many physical systems, however, the noise is multiplicatively coupled to the system, i.e., it amplifies with the signal itself. For example, turbulence in the combustors~\cite{clavin1994}, instrumental instabilities in nuclear magnetic resonance devices~\cite{granwehr2007}, pump fluctuation of dye lasers~\cite{fox1984,short1982}, and electrohydrodynamic instability in liquid crystals~\cite{brand1985} act as the source of the multiplicative noise, to name just a few. The effect of multiplicative noise on a system has been studied extensively over the past few decades, because such noise not only contaminates the signal but also affects the dynamical stability of the system~\cite{lieuwen2005, waugh2011}. Therefore, when analyzing the system influenced by the multiplicative noise, it is crucial to remove or suppress the effect of noise to unveil the original dynamics of the system.

In this paper, we propose a data-based modal decomposition algorithm that can accurately decompose a system affected by multiplicative noise. We combine the ideas of the conventional optimized DMD model and an image denoising model specific to multiplicative noise proposed in~\cite{AA:2008}, aiming to develop a novel optimized DMD model that is robust to multiplicative noise. Specifically, by closely following~\cite{AA:2008}, we construct a maximum a posteriori~(MAP) estimator for data-based modal decomposition of a linear dynamical system corrupted by gamma multiplicative noise. Because the constructed MAP estimator is complicated to solve numerically, an appropriate penalty relaxation technique should be applied to the MAP estimator to obtain the proposed optimized DMD model. We also present an efficient numerical algorithm to solve the proposed DMD model; we propose an alternating gradient descent method and suggest how to obtain a good initial guess for the algorithm. Convergence properties of the proposed alternating descent method are mathematically analyzed. Finally, we demonstrate the proposed DMD model on three numerical systems and show the reconstruction properties. We note that this paper is closely related to a recently published paper~\cite{Askham:2022} in the sense that both~\cite{Askham:2022} and this paper utilize statistical knowledge to design novel optimized DMD models that overcome particular difficulties.

This paper is organized as follows. In \cref{Sec:Preliminaries}, we provide preliminaries required for this paper, focusing on the optimized DMD model~\cite{AK:2018} and the Aubert--Aujol denoising model~\cite{AA:2008}. In \cref{Sec:Proposed}, we propose a novel optimized DMD model that is robust to multiplicative noise and investigate some mathematical properties of the proposed model. Next, an efficient numerical solver for the proposed model is considered in \cref{Sec:Algorithm}. Numerical results of the proposed model for various dynamic systems, including an engineering problem, are presented in \cref{Sec:Numerical}. Lastly, we conclude our paper with remarks in \cref{Sec:Conclusion}.

\section{Preliminaries}
\label{Sec:Preliminaries}
In this section, we introduce notations that are used throughout this paper. We also summarize key features of the optimized DMD model presented in~\cite{AK:2018} and some important variational models for noise removal~\cite{AA:2008,ROF:1992}. Motivated by the existing works~\cite{AK:2018,AA:2008}, we design a novel optimized DMD model that is robust to multiplicative noise in \cref{Sec:Proposed}.

\subsection{Notation}
As in~\cite{GV:2013}, we mostly use the standard notation accompanied with some MATLAB column notation. Let $A \in \mathbb{C}^{M \times N}$ be a complex matrix of size $M \times N$, and let $\xi \in \mathbb{C}^N$ be a complex vector of length $N$. We have the following list of notation associated with $A$ and $\xi$:

\begin{itemize}
\item $A_{ij}$ denotes the entry of $A$ in the $i$th row and $j$th column.
\item $A(i, :)$ denotes the $i$th row vector of $A$.
\item $A(:, j)$ denotes the $j$th column vector of $A$.
\item $A^{\rT}$ denotes the transpose of $A$.
\item $A^*$ denotes the Hermitian transpose of $A$.
\item $A^{\dag}$ denotes the Moore--Penrose pseudoinverse of $A$.
\item $\| A \|_F$ denotes the Frobenius norm of $A$.
\item $\| \xi \|_2$ denotes the $\ell^2$ norm of $\xi$.
\end{itemize}

As many matrices and vectors related to various physical dimensions will appear in this paper, it is convenient to use a unified notation for indices with respect to physical dimensions. In what follows, we use the indices $m$, $n$, and $r$ for the space, time, and rank dimensions, respectively.

\subsection{Optimized dynamic mode decomposition}
\label{Sec:additive}
For the sake of completeness, we present a brief summary of the optimized DMD model, following the presentation of Askham and Kutz~\cite{AK:2018}. Let $X = [\xi^1, \dots, \xi^N] \in \mathbb{R}^{M \times N}$ be a matrix of snapshots, where each $\xi^n = X(:, n) \in \mathbb{R}^M$, $1 \leq n \leq N$, represents the snapshot at time $t_n \in \mathbb{R}$~($t_1 < \dots < t_N$). We write $H = X^{\rT}$. Then the optimized DMD model is written as
\begin{equation}
\label{AK}
\min_{\alpha \in \mathbb{C}^R, \smallgap B \in \mathbb{C}^{R \times M}} \frac{1}{2} \left\| H - \Phi (\alpha) B \right\|_F^2,
\end{equation}
where $R \in \mathbb{Z}_{>0}$ is a target rank and $\Phi (\alpha) \in \mathbb{C}^{N \times R}$ is defined by
\begin{equation}
\label{Phi}
\Phi (\alpha)_{nr} = e^{\alpha_r t_n}, \quad 1 \leq n \leq N, \gap 1 \leq r \leq R.
\end{equation}
In order to investigate the meaning of~\cref{AK} in more details, we first write $B^{\rT} = [\beta^1, \dots, \beta^R]$, where $\beta^r \in \mathbb{C}^M$, $1 \leq r \leq R$. It follows that
\begin{equation*}
\frac{1}{2} \left\| H - \Phi (\alpha) B \right\|_F^2
= \frac{1}{2} \sum_{n=1}^N \| H(n, :) - \Phi (\alpha) (n, :) B \|_2^2
= \frac{1}{2} \sum_{n=1}^N \left\| \xi^n - \sum_{r=1}^R e^{\alpha_r t_n} \beta^r \right\|_2^2.
\end{equation*}
That is, the model~\cref{AK} finds an $\ell^2$-best approximation for $\xi^n$ of the form $\sum_{r=1}^R e^{\alpha_r t_n} \beta^r$. In what follows, we refer to~\cref{AK} as the $\ell^2$-optimized DMD model. Recall that, for a linear dynamic system
\begin{equation}
\label{dynamic}
\dot{z} (t) = A z (t), \quad z (0) = z^0
\end{equation}
with a diagonalizable matrix $A \in \mathbb{R}^{M \times M}$ and $z^0 \in \mathbb{R}^M$, the solution is given by
\begin{equation*}
z (t) = \sum_{m=1}^M e^{\alpha_m t} \beta^m,
\end{equation*}
where $\alpha_m \in \mathbb{C}$ is an eigenvalue of $A$ and $\beta^m \in \mathbb{C}^M$ is an eigenvector associated with $\alpha_m$, $1 \leq m \leq M$. In this sense, the $\ell^2$-optimized DMD model~\cref{AK} can be interpreted as follows: the entries of $\alpha$ approximate the $R$ dominant eigenvalues of a linear operator underlying the time series $X$, and the rows of $B$ approximate eigenvectors associated with the entries of $\alpha$.

One may eliminate the variable $B$ from~\cref{AK} by variable projection~\cite{GP:1973}. For a fixed $\alpha \in  \mathbb{C}^R$, the variable $B$ minimizing~\cref{AK} has a closed-form formula
\begin{equation}
\label{varpro}
B = \Phi (\alpha)^{\dag} H.
\end{equation}
If we substitute~\cref{varpro} into~\cref{AK}, then we obtain the following minimization problem whose variable is $\alpha$ only:
\begin{equation}
\label{AK_varpro}
\min_{\alpha \in \mathbb{C}^R} \frac{1}{2} \| H - \Phi (\alpha) \Phi (\alpha)^{\dag} H \|_F^2.
\end{equation}
For the equivalence relation between~\cref{AK} and~\cref{AK_varpro}, see~\cite[Theorem~2.1]{GP:1973}. Since~\cref{AK_varpro} is a nonlinear least squares problem and the dimension $R$ is not very big in general, a good strategy to solve~\cref{AK_varpro} is to use the Levenberg--Marquardt algorithm~\cite{AAAM:2011,Marquardt:1963}. One may refer to~\cite{AK:2018} for implementation details of the Levenberg--Marquardt algorithm for the $\ell^2$-optimized DMD model~\cref{AK_varpro}.

It was shown by numerical experiments in~\cite{AK:2018} that a strong point of the $\ell^2$-optimized DMD model is that it is more robust to additive noise than other existing DMD models such as~\cite{DHWR:2016,TRLBK:2014}. That is, the $\ell^2$-optimized DMD model results in more accurate eigenvalues and eigenvectors than the other models, even in the presence of additive noise of high variance. As another advantage of the model, because it can be regarded as a particular case of nonlinear fitting problem, it allows data collected at unevenly spaced sample times.

\subsection{Variational noise removal}
After a pioneering work of Rudin et al.~\cite{ROF:1992}, variational models have been broadly used in the field of signal and image processing for the purpose of denoising data. Here, we review several variational denoising models for image processing~\cite{AA:2008,ROF:1992}. Let $\Omega \subset \mathbb{R}^2$ be an image domain, and let $\cV$ be a suitable Banach space for digital images, e.g., $\cV = L^2 (\Omega)$. Suppose that we have a noisy image $f \in \cV$ and want to find a denoised counterpart $u \in \cV$. If we model the noise in $f$ as additive Gaussian noise, then we get the following linear inverse problem:
\begin{equation}
\label{inverse_add}
f = u + \epsilon,
\end{equation}
where the pointwise value of the noise $\epsilon$ is normally distributed with mean $0$ and variance $\sigma^2$ for some $\sigma > 0$. A popular approach to solve~\cref{inverse_add} is to find a MAP estimator for $u$~\cite{Getreuer:2012}. Noting that the conditional probability density $p(f|u)$ is given by
\begin{equation*}
p(f|u) = \frac{1}{\sqrt{2\pi} \sigma}  \exp \left( -\frac{1}{2\sigma^2} \intO (f-u)^2\,dx \right),
\end{equation*}
we can compute the MAP estimator for $u$ as follows:
\begin{equation} \begin{split}
\label{MAP_ROF}
\argmax_{u \in  \cV} p(u|f) &= \argmax_{u \in \cV} \frac{p(u) p(f|u)}{p(f)} \\
&= \argmin_{u \in \cV} \left\{ - \log p(u) - \log p(f|u) \right\} \\
&= \argmin_{u \in \cV} \left\{ \frac{1}{2\sigma^2} \intO (f-u)^2 \,dx + \phi (u) \right\},
\end{split} \end{equation}
where we used the Bayes' theorem in the first equality, and $\phi (u) = - \log p(u)$ is the prior on $u$, an a priori assumption on the likelihood of $u$. If we set $\phi(u)$ by the total variation of $u$~(see, e.g.,~\cite{LP:2020} for the definition of the total variation), then the last line of~\cref{MAP_ROF} becomes the celebrated Rudin--Osher--Fatemi model~\cite{ROF:1992}.
We note that other choices of $\phi (u)$ in~\cref{MAP_ROF} may yield denoising models of different purposes such as~\cite{BKP:2010,SKC:2003}.

Meanwhile, one may consider a situation that the noise in $f$ is multiplicative. We assume that
\begin{equation}
\label{inverse_mult}
f = u\epsilon,
\end{equation}
where $u > 0$ and the pointwise value of the noise $\epsilon$ follows the gamma distribution of mean $1$ and variance $\sigma^2$. Proceeding similarly to~\cref{MAP_ROF}, a MAP estimator for $u$ satisfying~\cref{inverse_mult} can be characterized as a solution of the following Aubert--Aujol model~\cite{AA:2008}:
\begin{equation}
\label{AA}
\min_{u \in \cV} \left\{ \frac{1}{\sigma^2} \int_{\Omega} \left( \log u + \frac{f}{u} \right) \,dx + \phi (u) \right\}.
\end{equation}
In the field of image processing, a typical choice for the prior $\phi (u)$ in~\cref{AA} is the total variation of $u$. Practical performance of the model~\cref{AA} for multiplicative noise removal can be found in~\cite{AA:2008,LNS:2010}.

As we have observed in~\cref{MAP_ROF,AA}, it is effective to use different data fidelity terms in denoising models for different kinds of noise. Several notable works \cite{CE:2005,LCA:2007,Nikolova:2004} have been on tailored data fidelity terms for various types of noise. Dependency of the quality of noise removal on data fidelity terms can be found in, e.g.,~\cite{LP:2020}.

\section{Proposed model}
\label{Sec:Proposed}
The purpose of this section is to propose a novel optimized DMD model that is robust to multiplicative noise. The essential idea of the construction of our proposed model is to combine the $\ell^2$-optimized DMD model~\cref{AK} and the Aubert--Aujol denoising model~\cref{AA}. In what follows, the indices $n$ and $m$ run from $1$ to $N$ and $M$, respectively.

First, we observe that a solution of~\cref{AK} can be regarded as a MAP estimator using a certain prior. Suppose that the matrix of observed snapshots $H$ in~\cref{AK} is expressed as the sum of a matrix $\tH \in \mathbb{R}^{N \times M}$ representing clean snapshots and a noise matrix $E \in \mathbb{R}^{N \times M}$ whose entries follow the normal distribution of mean $0$ and variance $\sigma^2$, i.e.,
\begin{equation}
\label{inverse_AK}
H_{nm} = \tH_{nm} + E_{nm}, \gap E_{nm} \sim N(0, \sigma^2 ).
\end{equation} 
For a fixed $R \in \mathbb{Z}_{>0}$, we define the set $\cD_R \subset \mathbb{R}^{N \times M}$ by
\begin{equation}
\label{D_R}
\cD_R = \left\{ K \in \mathbb{R}^{N \times M} : K = \Phi (\alpha) B \textrm{ for some } \alpha \in \mathbb{C}^R \textrm{, } B \in \mathbb{C}^{R \times M} \right\},
\end{equation}
where $\Phi (\alpha) \in \mathbb{C}^{N \times R}$ was defined in~\cref{Phi}.
Let $\chi_{\cD_R} \colon \mathbb{R}^{N \times M} \rightarrow \overline{\mathbb{R}}$ denote the characteristic function of $\cD_R$, i.e.,
\begin{equation*}
\chi_{\cD_R} (K) = \begin{cases}
0 & \textrm{ if } K \in \cD_R, \\
\infty & \textrm{ otherwise.}
\end{cases}
\end{equation*}
In DMD, we assume that dynamic features of the snapshots are determined by a few governing eigenvalues and eigenvectors of the dynamical system. In this perspective, a natural a priori assumption on $\tH$ is that $\tH$ belongs to the set $\cD_R$; we set
\begin{equation}
\label{DMD_prior}
- \log p(\tH) = \chi_{\cD_R} (\tH),
\end{equation}
with the convention $- \log 0 = \infty$. In the same manner as~\cref{MAP_ROF}, one can obtain the MAP estimator for $\tH$ as follows:
\begin{equation}
\label{MAP_AK}
\argmax_{\tH \in \mathbb{R}^{N \times M}} p(\tH|H)
= \argmin_{\tH \in \mathbb{R}^{N \times M}} \left\{ \frac{1}{2\sigma^2} \| H - \tH \|_F^2 + \chi_{\cD_R} (\tH) \right\}.
\end{equation}
In the minimization problem on the right-hand side of~\cref{MAP_AK}, one may drop the constant $\sigma^2$ since $\chi_{\cD_R}$ takes a value either $0$ or $\infty$. Invoking the definition~\cref{D_R} of the set $\cD_R$, we introduce two auxiliary variables $\alpha \in \mathbb{C}^R$ and $B \in \mathbb{C}^{R \times M}$ and set $\tH = \Phi (\alpha) B$. Then, the right-hand side of~\cref{MAP_AK} reduces to 
\begin{equation*}
\min_{\alpha \in \mathbb{C}^R, \smallgap B \in \mathbb{C}^{R \times M}} \frac{1}{2} \| H - \Phi (\alpha) B \|_F^2,
\end{equation*}
which is identical to the $\ell^2$-optimized DMD model~\cref{AK}. In conclusion, we have derived the $\ell^2$-optimized DMD model as the MAP estimator with the DMD prior~\cref{DMD_prior} for the denoising problem~\cref{inverse_AK}.

\begin{remark}
\label{Rem:improper}
In the prior assumption~\cref{DMD_prior} on $\tH$, it is not ensured that $p(\tH)$ is a probability density. However, thanks to the notion of improper prior, Bayesian analysis can be done successfully without assuming the prior knowledge is given by a probability density. One may refer to~\cite{LT:2018,TL:2010} for mathematically rigorous treatments on the notion of improper prior.
\end{remark}

Now, motivated by the observation that the $\ell^2$-optimized DMD model is a MAP estimator and the relation between the Rudin--Osher--Fatemi~\cite{ROF:1992} and Aubert--Aujol~\cite{AA:2008} models, we design an optimized DMD model that is suitable for data corrupted by multiplicative noise. Similar to~\cref{inverse_mult}, we consider the setting
\begin{equation}
\label{inverse_proposed}
H_{nm} = \tH_{nm} E_{nm},
\end{equation}
where each $E_{nm}$ follows the gamma distribution of mean $1$ and variance $\sigma^2$. Since $E_{nm}$ is always positive, $\tH_{nm}$ must have the same sign as $H_{nm}$. Hence, we may restrict the solution space for $\tH$ as the closed convex subset $\cS_{H}$ of $\mathbb{R}^{N \times M}$ defined by 
\begin{equation*}
\cS_{H} = \left\{ K \in \mathbb{R}^{N \times M} : K_{nm} H_{nm} \geq 0 \textrm{ if } H_{nm} \neq 0, \gap
K_{nm} = 0 \textrm{ if } H_{nm} = 0 \right\},
\end{equation*}
i.e., the collection of $K \in \mathbb{R}^{N \times M}$ such that $K_{nm}$ and $H_{nm}$ have the same sign.
Note that the value of $\tH_{nm}$ is determined by $0$ if $H_{nm} = 0$.
Invoking~\cite[Proposition~3.1]{AA:2008}, for $H_{nm} \neq 0$, we have
\begin{equation}
\label{gamma_density}
p ( H_{nm} | \tH_{nm}) = p \left( \frac{H_{nm}}{\tH_{nm}} \right) \frac{1}{|\tH_{nm}|}
= \frac{\beta^{\beta}}{|\tH_{nm}|^{\beta} \Gamma (\beta)} |H_{nm}|^{\beta - 1} \exp \left( - \frac{\beta H_{nm}}{\tH_{nm}} \right),
\end{equation}
where $\beta = 1/ \sigma^2$. Assuming all the entries of $H$ and $\tH$ are mutually independent, it follows by a similar argument to~\cref{MAP_ROF} that
\begin{equation}
\begin{split}
\label{MAP_proposed}
\argmax_{\tH \in \cS_H} p (\tH | H) &= \argmax_{\tH \in \cS_H} \prod_{H_{nm} \neq 0} p (\tH_{nm} | H_{nm}) \\
&= \argmin_{\tH \in \cS_H} \left\{ - \log p(\tH) - \sum_{H_{nm} \neq 0} \log p (H_{nm} | \tH_{nm}) \right\} \\
&= \argmin_{\tH \in \cS_H} \left\{\beta \sum_{H_{nm} \neq 0} \left( \log |\tH_{nm} | + \frac{H_{nm}}{\tH_{nm}} \right) + \chi_{\cD_R} (\tH) \right\} \\
&= \argmin_{\tH \in \cS_H} \left\{\sum_{H_{nm} \neq 0} \left( \log |\tH_{nm} | + \frac{H_{nm}}{\tH_{nm}} \right) + \chi_{\cD_R} (\tH) \right\}  ,
\end{split}
\end{equation}
where we used~\cref{DMD_prior,gamma_density} in the penultimate equality and dropped $\beta$ in the last equality. That is, the last line of~\cref{MAP_proposed}, which can equivalently be written as
\begin{equation}
\label{proposed_nonpenalty}
\min_{\alpha \in \mathbb{C}^R, \smallgap B \in \mathbb{C}^{R \times M}}
\sum_{H_{nm} \neq 0} \left( \log |(\Phi (\alpha) B)_{nm} | + \frac{H_{nm}}{(\Phi (\alpha) B)_{nm}} \right),
\end{equation}
is the MAP estimator with the DMD prior~\cref{DMD_prior} for the multiplicative denoising problem~\cref{inverse_proposed}.

Unfortunately, the structure of either the last line of~\cref{MAP_proposed} or~\cref{proposed_nonpenalty} is so complicated that it is difficult to design a suitable numerical solver for it. To simplify the model, we relax the $\chi_{\cD_R}(\tH)$-term by introducing a quadratic penalty term~\cite[section~1.A]{RW:2009}:
\begin{equation}
\label{proposed_penalty}
\min_{\tH \in \cS_H, \smallgap \alpha \in \mathbb{C}^R, \smallgap B \in \mathbb{C}^{R \times M}} \left\{ \sum_{H_{nm} \neq 0} \left( \log |\tH_{nm}| + \frac{H_{nm}}{\tH_{nm}} \right) + \frac{\eta}{2} \| \tH - \Phi (\alpha ) B \|_F^2 \right\},
\end{equation}
where $\eta > 0$ is a tunable parameter. In~\cref{proposed_penalty}, a minimizer with respect to $B$ for fixed $\tH \in \cS_H$ and $\alpha \in \mathbb{C}^R$ is given by $B = \Phi (\alpha)^{\dag} \tH$. Hence, similar to~\cref{AK_varpro}, the variable $B$ in~\cref{proposed_penalty} can be eliminated by variable projection~\cite{GP:1973} as follows:
\begin{equation}
\label{proposed_varpro}
\min_{\tH \in \cS_H, \smallgap \alpha \in \mathbb{C}^R}
\left\{ \cE ( \tH, \alpha ) := \sum_{H_{nm} \neq 0} \left( \log |\tH_{nm}| + \frac{H_{nm}}{\tH_{nm}} \right) + \frac{\eta}{2} \| \tH - \Phi (\alpha ) \Phi (\alpha)^{\dag} \tH \|_F^2 \right\}.
\end{equation}
The problem~\cref{proposed_varpro} is our proposed optimized DMD model.
Since the $\sum_{H_{nm} \neq 0}$-term in~\cref{proposed_varpro} has a similar form to the Aubert--Aujol model~\cref{AA}, $\tH$ is interpreted as a matrix of denoised snapshots free to multiplicative noise, obtained from the original data $H$.
Meanwhile, as $\alpha$ is yielded by minimizing the $\frac{\eta}{2}\|\cdot\|_F^2$-term similar to~\cref{AK_varpro}, it can be regarded as a good approximation for the DMD eigenvalues corresponding to the denoised data $\tH$.
In the remainder of this paper, we study mathematical and numerical aspects of the proposed model~\cref{proposed_varpro}.

\begin{remark}
\label{Rem:convention} 
One may notice that $\cE (\tH, \alpha)$ is not defined by the formula~\cref{proposed_varpro} if $\tH_{nm} = 0$ for any $n$ and $m$ such that $H_{nm} \neq 0$. In this case, we simply set $\cE (\tH, \alpha) = \infty$ in view of the following fact: for a nonzero real constant $a$, it satisfies that $\log |x| + a/x \rightarrow \infty$ if $x \rightarrow 0$ keeping the same sign as $a$.
\end{remark}

\begin{remark}
\label{Rem:unevenly}
A remarkable aspect of the $\ell^2$-optimized DMD model~\cref{AK_varpro} is that it allows unevenly spaced sample times~\cite{AK:2018}. By construction, the proposed model~\cref{proposed_varpro} naturally inherits such an advantage of the $\ell^2$-optimized DMD model and accommodates data collected at unevenly spaced sample times. We mention that DMD with unevenly spaced data has been considered as an important topic; see, e.g.,~\cite{GMP:2015,LC:2016}.
\end{remark}

\subsection{Mathematical study}
Due to the nature of DMD, optimized DMD models such as~\cref{AK_varpro,proposed_varpro} may admit nonunique global minimizers. The following example describes a situation when optimized DMD models have infinitely many global minimizers.

\begin{example}
\label{Ex:nonunique}
We take any $N \geq 2$, and set $M = R = 2$.
For $1 \leq n \leq N$, let $\xi^n = [0, 1]^{\rT}$ be a snapshot at time $t_n = 2n\pi$, i.e.,
\begin{equation*}
X = \begin{bmatrix} \xi^{1}, \dots, \xi^{N} \end{bmatrix}
= \begin{bmatrix}
0 & \dots & 0 \\
1 & \dots & 1
\end{bmatrix}
\in \mathbb{R}^{M \times N}.
\end{equation*}
We consider the linear dynamic system
\begin{equation}
\label{nonunique}
\dot{z}^k (t) = \begin{bmatrix} 0 & k \\ -k& 0 \end{bmatrix} z^k(t), \quad
z^k(0) = \begin{bmatrix} 0 \\ 1 \end{bmatrix},
\end{equation}
for $k \in \mathbb{Z}_{>0}$. It is easy to verify that the eigenvalues of the system matrix of~\cref{nonunique} are $\pm k i$ and that the solution $z^k (t)$ is given by
\begin{equation*}
z^k (t) = \begin{bmatrix} \sin k t \\ \cos k t \end{bmatrix}.
\end{equation*}
Hence, we have
\begin{equation*}
\xi^n = z^k (t_n), \quad 1 \leq n \leq N, \gap k \in \mathbb{Z}_{> 0}.
\end{equation*}
This implies that the $\ell^2$-optimized DMD model~\cref{AK_varpro} possesses infinitely many solutions $\alpha = [ki, -ki]^{\rT}$, $k \in \mathbb{Z}_{>0}$. Moreover, one can check that the proposed model~\cref{proposed_varpro} also admits infinitely many global minimizers $(\tH, \alpha) = (X^{\rT}, [ki, -ki]^{\rT})$, $k \in \mathbb{Z}_{>0}$.
\end{example}

\Cref{Ex:nonunique} implies that optimized DMD models such as~\cref{AK_varpro,proposed_varpro} may be noncoercive~(or not level-bounded), which makes variational analysis of the models notoriously difficult; note that the coercivity of a variational problem is a standard assumption in variational analysis to ensure the existence of a solution~\cite{RW:2009}. Here, to avoid such a difficulty, we deal with the \textit{localized} version of the proposed optimized DMD model~\cref{proposed_varpro} given by
\begin{equation}
\label{proposed_varpro_local}
\min_{\tH \in \cS_H, \smallgap \alpha \in \cA}
\left\{ \cE ( \tH, \alpha ) := \sum_{H_{nm} \neq 0} \left( \log |\tH_{nm}| + \frac{H_{nm}}{\tH_{nm}} \right) + \frac{\eta}{2} \| \tH - \Phi (\alpha ) \Phi (\alpha)^{\dag} \tH \|_F^2 \right\}
\end{equation}
instead of the original model, where $\cA$ is a closed and bounded subset of $\mathbb{C}^R$.
Intuitively,~\cref{proposed_varpro_local} corresponds to a situation when we have an additional piece of information on the bounds of physically meaningful DMD eigenvalues of~\cref{proposed_varpro}.
Before presenting an existence result for~\cref{proposed_varpro_local}, we need the following trivial fact.

\begin{lemma}
\label{Lem:log}
For $a \in \mathbb{R} \setminus \{ 0\}$, we define the function $g \colon \mathbb{R} \setminus \{ 0\} \rightarrow \mathbb{R}$ by
\begin{equation*}
g(x) = \log |x| + \frac{a}{x}, \quad x \in \mathbb{R} \setminus \{ 0\}.
\end{equation*}
Then we have the following:
\begin{enumerate}
\item If $a >0$, then the function $g(x)$~($x > 0$)  has the minimum $\log a + 1$ at $x = a$.
\item If $a < 0$, then the function $g(x)$~($x < 0$) has the minimum $\log (-a) + 1$ at $x = a$.
\end{enumerate}
\end{lemma}

Now, we have the following existence theorem for~\cref{proposed_varpro_local}.

\begin{proposition}
\label{Prop:exist}
The localized optimized DMD model~\cref{proposed_varpro_local} admits a solution, i.e.,  it has a global minimizer in $\cS_H \times \cA$.
\end{proposition}
\begin{proof}
We define a modified energy functional
\begin{equation}
\label{modified_energy}
\tcE (\tH, \alpha) =  \sum_{H_{nm} \neq 0} \left( \log |\tH_{nm}| + \frac{H_{nm}}{\tH_{nm}} - \log |H_{nm}| - 1 \right) + \frac{\eta}{2} \| \tH - \Phi (\alpha ) \Phi (\alpha)^{\dag} \tH \|_F^2.
\end{equation}
As $\tcE ( \tH, \alpha)$ and the proposed energy functional $\cE (\tH, \alpha)$ differ by a constant, minimizing $\tcE (\tH, \alpha)$ over $\cS_H \times \cA$ is equivalent to~\cref{proposed_varpro_local}; we consider the minimization problem for $\tcE (\tH, \alpha)$ instead of~\cref{proposed_varpro_local}.

Because $\tcE (\tH, \alpha)$ is nonnegative due to \cref{Lem:log}, we have
\begin{equation*}
\underline{\tcE} := \inf_{{\tH \in \cS_H, \smallgap \alpha \in \cA}} \tcE (\tH, \alpha) \geq 0.
\end{equation*}
We choose a sequence $\{( \tH^{(k)}, \alpha^{(k)} )\}_k$ in $\cS_H \times \cA$ such that $\lim_{k \rightarrow \infty} \tcE (\tH^{(k)}, \alpha^{(k)}) = \underline{\tcE}$ and $\tcE (\tH^{(k)}, \alpha^{(k)}) \leq \underline{\tcE} + 1$ for all $k$. It follows that
\begin{equation*}
\log |\tH_{nm}^{(k)}| + \frac{H_{nm}}{\tH_{nm}^{(k)}} - \log |H_{nm} | - 1 \leq \tcE (\tH^{(k)}, \alpha^{(k)}) \leq \underline{\tcE} + 1 
\end{equation*}
for all $n$ and $m$ such that $H_{nm} \neq 0$. Hence, the sequence $\{ \tH^{(k)} \}_k$ lies on the subset $\tcS_H$ of $\cS_H$ given by
\begin{equation*}
\tcS_H = \left\{ K \in \cS_H : \log |K_{nm}| + \frac{H_{nm}}{K_{nm}} \leq \log |H_{nm}| + \underline{\tcE} + 2 \textrm{ if } H_{nm} \neq 0 \right\}.
\end{equation*}
As both $\tcS_H$ and $\cA$ are closed and bounded, we can ensure that a limit point $(\tH^{(\infty)}, \alpha^{(\infty)})$ of the sequence $\{( \tH^{(k)}, \alpha^{(k)} )\}_k$ belongs to $\tcS_H \times \cA$. By the continuity of $\tcE (\tH, \alpha)$, we get $\tcE (\tH^{(\infty)}, \alpha^{(\infty)}) = \underline{\tcE}$. That is, $(\tH^{(\infty)}, \alpha^{(\infty)})$ is a global minimizer of $\tcE (\tH, \alpha)$ in $\cS_H \times \cA$.
\end{proof}

If we localize~\cref{proposed_penalty} in the same manner as~\cref{proposed_varpro_local}, then we get \begin{equation}
\label{proposed_penalty_local}
\min_{\tH \in \cS_H, \smallgap \alpha \in \cA, \smallgap B \in \mathbb{C}^{R \times M}} \left\{ \sum_{H_{nm} \neq 0} \left( \log |\tH_{nm}| + \frac{H_{nm}}{\tH_{nm}} \right) + \frac{\eta}{2} \| \tH - \Phi (\alpha ) B \|_F^2 \right\}.
\end{equation}
We observe how two minimization problems~\cref{proposed_varpro_local,proposed_penalty_local} are related. \Cref{Prop:equiv} summarizes the equivalence relation between~\cref{proposed_varpro_local,proposed_penalty_local}. As it can be shown by almost the same argument as~\cite[Proposition~3.1]{LPP:2019}, we omit its proof.

\begin{proposition}
\label{Prop:equiv}
If $(\tH, \alpha, B) \in \cS_H \times \cA \times \mathbb{C}^{R \times M}$ is a solution of~\cref{proposed_penalty_local}, then $(\tH, \alpha)$ is a solution of~\cref{proposed_varpro_local}.
Conversely, if $(\tH, \alpha) \in \cS_H \times \cA$ is a solution of~\cref{proposed_varpro_local}, then $(\tH, \alpha, \Phi (\alpha)^{\dag} \tH)$ is a solution of~\cref{proposed_penalty_local}.
\end{proposition}

Thanks to \cref{Prop:equiv}, we may refer to~\cref{proposed_penalty_local} as the proposed model as well. Recall that~\cref{proposed_penalty_local} is not a MAP estimator for the denoising problem~\cref{inverse_proposed} but a penalty relaxation of that. Therefore, it is necessary to analyze how well a solution of~\cref{proposed_penalty_local} approximates a solution of the exact MAP estimator
\begin{equation}
\label{proposed_nonpenalty_local}
\min_{\alpha \in \cA, \smallgap B \in \mathbb{C}^{R \times M}}
\sum_{H_{nm} \neq 0} \left( \log |(\Phi (\alpha) B)_{nm} | + \frac{H_{nm}}{(\Phi (\alpha) B)_{nm}} \right),
\end{equation}
i.e., the localized version of~\cref{proposed_nonpenalty}. The following proposition says that the proposed DMD model~\cref{proposed_penalty_local} acts like~\cref{proposed_nonpenalty_local} if the penalty parameter $\eta$ is sufficiently large, while its behavior becomes similar to the localized $\ell^2$-optimized DMD model
\begin{equation}
\label{AK_local}
\min_{\alpha \in \cA, \smallgap B \in \mathbb{C}^{R \times M}} \frac{1}{2} \left\| H - \Phi (\alpha) B \right\|_F^2
\end{equation}
if $\eta$ approaches to $0$.

\begin{proposition}
\label{Prop:gamma}
For $\eta > 0$, let $(\tH^{\eta}, \alpha^{\eta})$ denote a solution of the localized optimized DMD model~\cref{proposed_varpro_local} such that all the entries of $\alpha^{\eta}$ are distinct. Then we have the following:
\begin{enumerate}
\item Assume that $\{ (\tH^{\eta}, \alpha^{\eta} )\}_{\eta}$ has a limit point $(\tH^{\infty}, \alpha^{\infty})$ as $\eta \rightarrow \infty$ such that all the entries of $\alpha^{\infty}$ are distinct.
Then $(\alpha^{\infty}, \Phi (\alpha^{\infty})^{\dag} \tH^{\infty})$ solves~\cref{proposed_nonpenalty_local}.
\item Assume that $\{ (\tH^{\eta}, \alpha^{\eta} )\}_{\eta}$ has a limit point $(\tH^{0}, \alpha^{0})$ as $\eta \rightarrow 0$ such that all the entries of $\alpha^{0}$ are distinct.
Then $(\alpha^0, \Phi (\alpha^0)^{\dag} \tH^0 )$ solves~\cref{AK_local}.
\end{enumerate}
\end{proposition}
\begin{proof}
Throughout this proof, the energy functional of~\cref{proposed_penalty_local} is denoted by $\cE_*^{\eta}(\tH, \alpha, B)$. By \cref{Prop:equiv}, for each $\eta > 0$, $(\tH^{\eta}, \alpha^{\eta}, \Phi (\alpha^{\eta})^{\dag} \tH^{\eta} )$ solves~\cref{proposed_penalty_local}, i.e., it is a global minimizer of $\cE_*^{\eta} (\tH, \alpha, B)$.
\begin{enumerate}
\item Since the entries of $\alpha^{\eta}$~($\eta > 0$) and $\alpha^{\infty}$ are distinct, the matrices $\Phi (\alpha^{\eta})$ and $\Phi (\alpha^{\infty})$ have full column rank. Invoking~\cite[Theorem~2.1]{Wei:1990}, we deduce that $(\tH^{\infty}, \alpha^{\infty}, \Phi (\alpha^{\infty})^{\dag} \tH^{\infty})$ is a limit point of  $\{ (\tH^{\eta}, \alpha^{\eta}, \Phi (\alpha^{\eta} )^{\dag} \tH^{\eta}) \}_{\eta}$ as $\eta$ tends to infinity. Meanwhile, we observe that $\cE_*^{\eta} (\tH, \alpha, B)$ increases pointwise to
\begin{equation*}
\cE_*^{\infty} (\tH, \alpha, B) = \sum_{H_{nm} \neq 0} \left( \log |\tH_{nm}| + \frac{H_{nm}}{\tH_{nm}} \right) + \chi_{\{ \tH = \Phi (\alpha) B \}} (\tH, \alpha, B)
\end{equation*}
as $\eta \rightarrow \infty$, where
\begin{equation*}
\chi_{\{ \tH = \Phi (\alpha) B \}} (\tH, \alpha, B) = \begin{cases}
0 & \textrm{ if } \tH  = \Phi (\alpha) B, \\
\infty & \textrm{ otherwise.}
\end{cases}
\end{equation*}
By~\cite[Proposition~7.4(d)]{RW:2009}, $\cE_*^{\eta} (\tH, \alpha, B)$ epi-converges to $\cE_*^{\infty} (\tH, \alpha, B)$ when $\eta \rightarrow \infty$. Hence, we conclude by~\cite[Theorem~7.33]{RW:2009} that the limit point $(\tH^{\infty}, \alpha^{\infty}, \Phi (\alpha^{\infty})^{\dag} \tH^{\infty})$ minimizes $\cE_*^{\infty} (\tH, \alpha, B)$. Equivalently, it solves~\cref{proposed_nonpenalty_local}.

\item Similar to~\cref{modified_energy}, we define
\begin{equation*}
\tcE_*^{\eta} ( \tH, \alpha, B) = \frac{1}{\eta} \sum_{H_{nm \neq 0}} \left( \log |\tH_{nm} | + \frac{H_{nm}}{\tH_{nm}} - \log |H_{nm}| - 1 \right) + \frac{1}{2} \| \tH -\Phi (\alpha) B \|_F^2.
\end{equation*}
Here, the $\sum_{H_{nm} \neq 0}$-term vanishes if and only if $\tH = H$. Invoking \cref{Prop:equiv} and~\cite[Theorem~2.1]{Wei:1990}, we deduce that $(\tH^{\eta}, \alpha^{\eta}, \Phi (\alpha^{\eta})^{\dag} \tH^{\eta})$ minimizes $\tcE_*^{\eta} (\tH, \alpha, B)$ and it accumulates at $(\tH^0, \alpha^0, \Phi (\alpha^{0})^{\dag} \tH^{0})$ when $\eta$ tends to $0$. Observe that as $\eta \rightarrow 0$, $\tcE_*^{\eta} (\tH, \alpha, B)$ increases pointwise to
\begin{equation*}
\tcE_*^0 (\tH, \alpha, B) = \frac{1}{2} \| \tH - \Phi (\alpha) B \|_F^2 + \chi_{\{ \tH = H \}} (\tH),
\end{equation*}
where 
\begin{equation*}
\chi_{\{ \tH = H \}} (\tH) = \begin{cases}
0 & \textrm{ if } \tH  = H, \\
\infty & \textrm{ otherwise.}
\end{cases}
\end{equation*}
Again by~\cite[Proposition~7.4(d) and Theorem~7.33]{RW:2009}, $\tcE_*^{\eta} (\tH, \alpha, B)$ epi-converges to $\tcE_*^0 (\tH, \alpha, B)$ when $\alpha \rightarrow 0$, which implies that $(\tH^0, \alpha^0, \Phi (\alpha^{0})^{\dag} \tH^{0})$ is a global minimizer of $\tcE_*^0 (\tH, \alpha, B)$. Therefore, $(\tH^0, \alpha^0, \Phi (\alpha^{0})^{\dag} \tH^{0})$ is a solution of~\cref{AK_local}.
\end{enumerate}
\end{proof}

\begin{remark}
\label{Rem:repeated}
In \cref{Prop:gamma}, we assumed that any $\alpha$ obtained as a solution of the proposed DMD model~\cref{proposed_varpro_local} has no repeated entries. Such an assumption is natural since any algorithms for the optimized DMD models such as~\cref{AK_local,proposed_varpro_local}  suffer from numerical instability near a solution if the matrix $\Phi (\alpha)$ is rank-deficient. In addition, it was mentioned in~\cite[Remark~2]{AK:2018} it is difficult to approximate dynamics using exponentials alone when a system matrix is not diagonalizable.
\end{remark}

\Cref{Prop:gamma} displays a favorable aspect of the proposed model.
By tuning the penalty parameter $\eta$ in~\cref{proposed_varpro_local} appropriately, the proposed model can not only approximate the MAP estimator~\cref{proposed_nonpenalty_local} well, but also it can lie between two limiting models~\cref{proposed_nonpenalty_local,AK_local} and inherit some advantages of them. For instance, the proposed DMD model is expected to be robust on not only the multiplicative noise but also the mixture of additive and multiplicative noise. Numerical results that verify the effectiveness of the proposed model on a realistic physical system corrupted by mixed additive and multiplicative noise will be presented in \cref{Subsec:Combustor}.

\section{Algorithm}
\label{Sec:Algorithm}
This section is devoted to an efficient numerical algorithm to solve the proposed optimized DMD model~\cref{proposed_varpro}. Since~\cref{proposed_varpro} has two sets of unknowns $\tH$ and $\alpha$, a natural idea to solve~\cref{proposed_varpro} is to adopt an alternating descent method~\cite{Beck:2015,BT:2013,BST:2014}. In addition, due to the nonconvexity of the energy functional $\cE (\tH, \alpha)$, a good initialization strategy should be considered in order to expect good performance~\cite{AK:2018,YL:2020}. 

\subsection{Alternating descent method}
To implement an alternating descent method for the proposed model~\cref{proposed_varpro}, we should be able to compute the $\tH$-gradient $\nabla_{\tH} \cE$ and $\alpha$-gradient $\nabla_{\alpha} \cE$ of the energy functional $\cE (\tH, \alpha)$. Note that $\tH$ is a real matrix while $\alpha$ is a complex vector. Since both real and complex variables appear in~\cref{proposed_varpro}, some careful considerations are required in the computation of gradients. We first state an elementary lemma on the differentiation of a norm of a complex matrix with respect to a real matrix.

\begin{lemma}
\label{Lem:real_grad}
For $A \in \mathbb{C}^{L \times N}$, we define the functional $\cJ \colon \mathbb{R}^{N \times M} \rightarrow \mathbb{R}$ by
\begin{equation*}
\cJ (K) = \frac{1}{2} \| AK \|_F^2.
\end{equation*}
Then we have
\begin{equation*}
\nabla_K \cJ (K) = \Re (A^{*} AK).
\end{equation*}
\end{lemma}
\begin{proof}
Let $U = \Re (A) \in \mathbb{R}^{L \times N}$ and $V = \Im (A) \in \mathbb{R}^{L \times N}$, i.e., $A = U + iV$. Then we have
\begin{equation*}
A^{*} A K = (U^{\rT} - iV^{\rT})(U + iV) K = (U^{\rT} U + V^{\rT} V) K + i(U^{\rT} V - V^{\rT} U) K,
\end{equation*}
so that $\Re (A^{*} A K) = (U^{\rT} U + V^{\rT} V) K$. Meanwhile, since
\begin{equation*}
\| A K \|_F^2 = \| UK + iVK \|_F^2 = \| UK \|_F^2 + \| VK \|_F^2, 
\end{equation*}
it follows that
\begin{equation*}
\nabla_K \cJ (K) = \frac{1}{2} \nabla_K \left( \| UK \|_F^2 + \| VK \|_F^2 \right) = (U^{\rT} U + V^{\rT} V)K.
\end{equation*}
Therefore, we conclude that $\nabla_K \cJ (K) = \Re (A^{*} A K)$.
\end{proof}

When we differentiate the energy functional $\cE (\tH, \alpha)$ with respect to $\tH$, the $\sum_{H_{nm \neq 0}}$-term can be handled entrywise by the single variable calculus; we have
\begin{equation*}
\nabla_{\tH} \left(\sum_{H_{nm} \neq 0} \left( \log |\tH_{nm}| + \frac{H_{nm}}{\tH_{nm}} \right) \right) = D (\tH),
\end{equation*}
where the matrix $D(\tH) \in \mathbb{R}^{N \times M}$ is given by
\begin{equation*}
D (\tH)_{nm} = \begin{cases} \displaystyle
\frac{1}{\tH_{nm}} - \frac{H_{nm}}{\tH_{nm}^2} & \textrm{ if } H_{nm} \neq 0, \\
0 & \textrm{ otherwise.}
\end{cases}
\end{equation*}
On the other hand, invoking \cref{Lem:real_grad}, the differentiation of $\| \tH - \Phi (\alpha) \Phi (\alpha)^{\dag} \tH \|_F^2$ is straightforward:
\begin{equation*}
\frac{1}{2} \nabla_{\tH} \left( \| \tH - \Phi (\alpha) \Phi (\alpha)^{\dag} \tH \|_F^2 \right)
= (I - \Phi (\alpha) \Phi(\alpha)^{\dag})^{*} (I - \Phi (\alpha) \Phi(\alpha)^{\dag}) \tH = (I - \Phi (\alpha) \Phi(\alpha)^{\dag}) \tH,
\end{equation*}
where the last equality is due to the fact that $I - \Phi (\alpha) \Phi(\alpha)^{\dag}$ is an orthogonal projection~\cite{AK:2018}. In summary, we get the following formula for $\nabla_{\tH} \cE$:
\begin{equation}
\label{grad_tH}
\nabla_{\tH} \cE (\tH, \alpha) = D(\tH) + \eta \Re ( \tH - \Phi (\alpha) \Phi (\alpha)^{\dag} \tH ).
\end{equation}
Note that $\cE (\tH, \alpha)$ is not differentiable if $\tH_{nm} = 0$ for some $n$ and $m$ such that $H_{nm} \neq 0$. We define the subset $\cS_H^{\circ}$ of $\cS_H$ as
\begin{equation*}
\cS_{H}^{\circ} = \left\{ K \in \mathbb{R}^{N \times M} : K_{nm} H_{nm} > 0 \textrm{ if } H_{nm} \neq 0, \gap
K_{nm} = 0 \textrm{ if } H_{nm} = 0 \right\},
\end{equation*}
so that $\cE (\tH, \alpha)$ is differentiable at every element in $\cS_H^{\circ}$. \Cref{Rem:convention} implies that $\cE (\tH, \alpha) = \infty$ in $\cS_H \setminus \cS_H^{\circ}$.

Next, we focus on the differentiation of $\cE (\tH, \alpha)$ with respect to $\alpha$. Because $\cE (\tH, \alpha)$ is nonconstant real-valued, it is not complex differentiable in the standard sense. That is, it is impossible to design gradient-based optimization algorithms for this kind of problems with the standard complex calculus. We notice that, in several existing works~\cite{AAAM:2011,LA:2008}, the $\mathbb{CR}$-calculus~\cite{Kreutz:2009} has been successfully adopted to design gradient-based algorithms instead of the standard complex calculus. Similar to~\cite{Kreutz:2009}, we define the complex gradient $\nabla_{\alpha} \cE (\tH, \alpha)$ by
\begin{equation}
\label{complex_gradient}
\nabla_{\alpha} \cE (\tH, \alpha) = 2 \left( \frac{\partial \cE}{\partial \alpha} (\tH, \alpha ) \right)^{*},
\end{equation}
where $\partial \cE / \partial \alpha$ denotes the Wirtinger derivative of $\cE $ with respect to $\alpha$, i.e., the derivative with respect to $\alpha$ with $\bar{\alpha}$ held constant. It was shown in~\cite[Proposition~2]{LA:2008} that the definition~\cref{complex_gradient} of the complex gradient agrees with that of the real gradient. By the chain rule for the Wirtinger derivatives~\cite{Kreutz:2009}, we have
\begin{equation*}
\frac{\partial \cE}{\partial \alpha} (\tH, \alpha ) = \frac{\eta}{2} ( \tH - \Phi (\alpha) \Phi (\alpha)^{\dag} \tH)^{*} J (\tH, \alpha),
\end{equation*}
where $J (\tH, \alpha)$ is the Jacobian matrix of $\tH - \Phi (\alpha) \Phi (\alpha)^{\dag} \tH$ with respect to $\alpha$. Hence, we get a formula for $\nabla_{\alpha} \cE (\tH, \alpha)$ as follows:
\begin{equation}
\label{grad_alpha}
\nabla_{\alpha} \cE (\tH, \alpha) = \eta J (\tH, \alpha)^{*} (\tH - \Phi (\alpha) \Phi (\alpha)^{\dag} \tH).
\end{equation}
One may refer to~\cite[section~2.2]{AK:2018} for a detailed explanation on how to compute $J (\tH, \alpha)$.

Now, we are ready to propose an alternating descent method for the proposed DMD model~\cref{proposed_varpro}; see \cref{Alg:proposed}.

\begin{algorithm}
\caption{Alternating descent method for the proposed model~\cref{proposed_varpro}}
\begin{algorithmic}[]
\label{Alg:proposed}
\STATE Choose $\tH^{(0)} \in \cS_{H}^{\circ}$, $\alpha^{(0)} \in \mathbb{C}^R$, $\tau_{\tH} > 0$, and $\tau_{\alpha} > 0$.

\FOR{$k=0,1,2,\dots$}
\STATE \textbf{(Projected gradient descent with respect to $\tH$)}
\STATE $$\tau_{\tH} \leftarrow 2\tau_{\tH}$$
\REPEAT
\STATE $$\tH^{(k+1)} = \proj_{\cS_H} \left( \tH^{(k)} - \tau_{\tH} \nabla_{\tH} \cE (\tH^{(k)}, \alpha^{(k)}) \right)
\quad \textrm{(see~\cref{grad_tH})}$$
\IF {$\displaystyle \cE (\tH^{(k+1)}, \alpha^{(k)}) > \cE (\tH^{(k)}, \alpha^{(k)}) - \frac{1}{2 \tau_{\tH}} \| \tH^{(k+1)} - \tH^{(k)} \|_F^2$ }
\STATE $$\tau_{\tH} \leftarrow \tau_{\tH} / 2$$
\ENDIF
\UNTIL {$\displaystyle \cE (\tH^{(k+1)}, \alpha^{(k)}) \leq \cE (\tH^{(k)}, \alpha^{(k)}) - \frac{1}{2 \tau_{\tH}} \| \tH^{(k+1)} - \tH^{(k)} \|_F^2$}

\STATE \textbf{(Gradient descent with respect to $\alpha$)}
\STATE $$\tau_{\alpha} \leftarrow 2\tau_{\alpha}$$
\REPEAT
\STATE $$\alpha^{(k+1)} = \alpha^{(k)} - \tau_{\alpha} \nabla_{\alpha} \cE (\tH^{(k+1)}, \alpha^{(k)})
\quad \textrm{(see~\cref{grad_alpha})}$$
\IF {$\displaystyle \cE (\tH^{(k+1)}, \alpha^{(k+1)}) > \cE (\tH^{(k+1)}, \alpha^{(k)}) - \frac{1}{2\tau_{\alpha} } \| \alpha^{(k+1)} - \alpha^{(k)} \|_2^2$}
\STATE $$\tau_{\alpha} \leftarrow \tau_{\alpha} / 2$$
\ENDIF
\UNTIL {$\displaystyle \cE (\tH^{(k+1)}, \alpha^{(k+1)}) \leq \cE (\tH^{(k+1)}, \alpha^{(k)}) - \frac{1}{2\tau_{\alpha} } \| \alpha^{(k+1)} - \alpha^{(k)} \|_2^2$}
\ENDFOR
\end{algorithmic}
\end{algorithm}

\Cref{Alg:proposed} is composed of gradient descent steps with respect to $\tH$ and $\alpha$ described in~\cref{grad_tH} and~\cref{grad_alpha}, respectively. To ensure that $\tH^{(k)}$, $k \geq 0$, always belongs to $\cS_H$, we employ the projected gradient descent for $\tH$. Moreover, an intial guess $\tH^{(0)}$ for $\tH$ is chosen such that $\tH^{(0)} \in \cS_H^{\circ}$ since the energy functional is not differentiable with respect to $\tH$ on $\cS_H \setminus \cS_H^{\circ}$. It is easy to verify that $\tH^{(k)}$ remains in $\cS_H^{\circ}$ for all $k \geq 0$ if $\tH^{(0)} \in \cS_H^{\circ}$. Meanwhile, \cref{Alg:proposed} has backtracking steps to determine suitable step sizes $\tau_{\tH}$ and $\tau_{\alpha}$ in each iteration. We note that there have been a number of existing works on applications of backtracking strategies for mathematical optimization; see, e.g.,~\cite{BT:2013,CC:2019,Park:2022}. In particular, the backtracking strategy used in \cref{Alg:proposed} is motivated by the block coordinate descent method with backtracking proposed in~\cite{BT:2013}.

We make a brief discussion on the computational cost of each iteration of \cref{Alg:proposed}. The projection $\proj_{\cS_H}$ onto the set $\cS_H$ can be done cheaply by using the following formula:
\begin{equation*}
\left( \proj_{\cS_H} (K) \right)_{nm} = \begin{cases}
\max \{ K_{nm}, 0 \} & \textrm{ if } H_{nm} > 0, \\
0 & \textrm{ if } H_{nm} = 0, \\
\min \{ K_{nm}, 0 \} & \textrm{ if } H_{nm} < 0.
\end{cases}
\end{equation*}
Computation of the gradients $\nabla_{\tH} \cE (\tH^{(k)}, \alpha^{(k)})$ and $\nabla_{\alpha} \cE (\tH^{(k+1)}, \alpha^{(k)})$ is required only once at the $k$th iteration of \cref{Alg:proposed}. Once the gradients are evaluated, the gradients can be stored in memory and used in the backtracking steps without additional computation. Similarly, in the gradient descent steps for $\tH$ and $\alpha$, it is enough to evaluate the energy values $\cE (\tH^{(k)}, \alpha^{(k)})$ and $\cE (\tH^{(k+1)}, \alpha^{(k)})$ only once, respectively. The computational cost of backtracking steps is marginal. At each inner iteration in the backtracking step for searching the step size $\tau_{\tH}$, the required computations are a single evaluation of the energy value $\cE (\tH^{(k+1)}, \alpha^{(k)})$ and some minor scalar operations; the values of $\cE (\tH^{(k)}, \alpha^{(k)})$ and $\| \nabla_{\tH} \cE (\tH^{(k)}, \alpha^{(k)}) \|_F$ are computed before the backtracking step and stored as explained above. A similar explanation can be made for each iteration in the backtracking step for $\tau_{\alpha}$.

\subsection{Convergence analysis}
In order to ensure the robustness of the proposed algorithm, it is crucial to observe the asymptotic behavior of \cref{Alg:proposed} as the iteration count $k$ tends to infinity. First, we present an elementary fact on functions with locally Lipschitz gradients; although \cref{Lem:Lipschitz} shows the result for functions with complex domains, the same result holds for functions with real domains~\cite[Lemma~2.3]{BT:2009}.

\begin{lemma}
\label{Lem:Lipschitz}
Let $\cS$ be a closed convex subset of $\mathbb{C}^N$. Assume that a function $g \colon \mathbb{C}^N \rightarrow \mathbb{R}$ has the locally Lipschitz complex gradient. That is, for any $x \in \mathbb{C}^N$, there exists a neighborhood $\cN_x \subset \mathbb{C}^N$ of $x$ and a constant $L_{x} > 0$ such that
\begin{equation}
\label{Lipschitz}
\| \nabla g(y) - \nabla g(x) \|_2 \leq L_x \| y - x \|_2, \quad y \in \cN_x,
\end{equation}
where $\nabla g$ is the complex gradient of $g$ defined as~\cref{complex_gradient}. Then, for any $x \in \cS$ and $\tau \in (0, 1/L_x]$ satisfying $z = \proj_{\cS} \left( x - \tau \nabla g(x) \right) \in \cN_x$, we have
\begin{equation*}
g(z ) \leq g(x) - \frac{1}{2\tau} \| z - x \|_2^2.
\end{equation*}
\end{lemma}
\begin{proof}
For a fixed $x \in \mathbb{C}^N$, we take $\cN_x \subset \mathbb{C}^N$ and $L_x > 0$ as given in~\cref{Lipschitz}. Take any $y \in \cN_x$. If we define a function $G \colon [0,1] \subset \mathbb{R} \rightarrow \mathbb{R}$ by
\begin{equation*}
G(t) = g((1-t)x + ty), \quad t \in [0, 1],
\end{equation*}
and set $z(t) = (1-t)x + ty$, then we get
\begin{equation*} \begin{split}
G'(t) &= \frac{\partial g}{\partial z} z'(t) + \frac{\partial g}{\partial \bar{z}} \bar{z}' (t) \\
&=2 \Re \left( \frac{\partial g}{\partial z} (y-x) \right) = \Re \left( \nabla g (z(t))^{*} (y-x) \right).
\end{split} \end{equation*}
It follows that
\begin{equation*} \begin{split}
g(y) - g(x) &= \int_0^1 G'(t) \,dt \\
&= \Re \left( \int_0^1 \left( \nabla g( (1-t)x + ty) - \nabla g(x) \right)^{*} (y-x) \,dt \right) + \Re \left( \nabla g(x)^{*} (y-x) \right) \\
&\leq \int_0^1 L_x t  \|y-x\|_2^2 \,dt + \Re \left( \nabla g(x)^{*} (y-x) \right) \\
&= \frac{L_x}{2} \|y-x\|_2^2 + \Re \left( \nabla g(x)^{*} (y-x) \right).
\end{split} \end{equation*}
That is, we obtain
\begin{equation}
\label{descent_lemma}
g(y) \leq g(x) + \Re \left( \nabla g(x)^{*} (y-x) \right) + \frac{L_x}{2} \|y - x\|_2^2.
\end{equation}

Now, we assume that $x \in \cS$ and that $z = \proj_{\cS} \left( x- \tau \nabla g(x) \right) \in \cN_x$. The elementary property of the projection presented in~\cite[Theorem~5.2]{Brezis:2010} implies that
\begin{equation*}
\Re \left( \left( z - (x - \tau \nabla g(x)) \right)^{*} (z-x) \right) \leq 0.
\end{equation*}
We substitute $y$ in~\cref{descent_lemma} by $z$.
Then we get
\begin{equation*} \begin{split}
g(z) &\leq g(x) + \Re \left( \nabla g(x)^{*} (z-x) \right) + \frac{L_x}{2} \|z - x \|_2^2 \\
&\leq  g(x) + \Re \left( \nabla g(x)^{*} (z-x) \right) + \frac{L_x}{2} \|z - x \|_2^2 - \frac{1}{\tau} \Re \left( \left( z - (x - \tau \nabla g(x)) \right)^{*} (z-x) \right) \\
&= g(x) + \left( \frac{L_x}{2} - \frac{1}{\tau} \right) \| z - x \|_2^2 \\
&\leq g(x) - \frac{1}{2\tau} \| z - x \|_2^2,
\end{split} \end{equation*}
which completes the proof.
\end{proof}

Using \cref{Lem:Lipschitz}, we can prove that the backtracking processes for the step sizes $\tau_{\tH}$ and $\tau_{\alpha}$ in \cref{Alg:proposed} are always finite.

\begin{proposition}
\label{Prop:backtracking}
In \cref{Alg:proposed}, assume that the entries of each $\alpha^{(k)}$, $k \geq 0$, are distinct. Then the backtracking processes for the step sizes $\tau_{\tH}$ and $\tau_{\alpha}$ terminate in finite steps.
\end{proposition}
\begin{proof}
Observing that $\Phi (\alpha)$ is infinitely many differentiable with respect to $\alpha$ and that each $\Phi (\alpha^{(k)})$ has full column rank, one can verify that $\nabla_{\tH} \cE (\tH, \alpha)$ and $\nabla_{\alpha} \cE (\tH, \alpha)$ given in~\cref{grad_tH} and~\cref{grad_alpha} are continuously differentiable with respect to $\tH$ and $\alpha$ in $\cS_H^{\circ} \times \mathbb{C}^R$, respectively~(see~\cite{GP:1973} for the differentiability of $\Phi (\alpha)^{\dag}$). Hence, they are locally Lipschitz. Invoking \cref{Lem:Lipschitz}, we conclude that at the $k$th iteration of \cref{Alg:proposed}, the backtracking process for $\tau_{\tH}$ terminates when $\tau_{\tH}$ becomes sufficiently small so that
\begin{equation*}
\tau_{\tH} \leq 1/ L_{\tH^{(k)}} \textrm{ and } \proj_{\cS_H} \left( \tH^{(k)} - \tau_{\tH} \nabla_{\tH} \cE (\tH^{(k)}, \alpha^{(k)}) \right) \in \cN_{\tH^{(k)}},
\end{equation*} where $L_{\tH^{(k)}}$ is a local Lipschitz constant of $\nabla_{\tH} \cE ( \cdot, \alpha^{(k)})$ in a neighborhood $\cN_{\tH^{(k)}}$ of $\tH^{(k)}$ defined in~\cref{Lipschitz}.
Similarly, the backtracking process for $\tau_{\alpha}$ terminates when
\begin{equation*}
\tau_{\alpha} \leq 1/ L_{\alpha^{(k)}} \textrm{  and } \alpha^{(k)} - \tau_{\alpha} \nabla_{\alpha} \cE (\tH^{(k+1)},\alpha^{(k)}) \in \cN_{\alpha^{(k)}},
\end{equation*}
where $L_{\alpha^{(k)}}$ is a local Lipschitz constant of $\nabla_{\alpha} \cE (\tH^{(k+1)}, \cdot)$ in a neighborhood $\cN_{\alpha^{(k)}}$ of $\alpha^{(k)}$.
\end{proof}

Finally, we present a monotone convergence property of \cref{Alg:proposed} in \cref{Prop:convergence}; it is guaranteed that the energy values corresponding to the sequence $\{ (\tH^{(k)}, \alpha^{(k)}) \}_k$ generated by \cref{Alg:proposed} always decreases when $k$ grows up to infinity.

\begin{proposition}
\label{Prop:convergence}
In \cref{Alg:proposed}, assume that the entries of each $\alpha^{(k)}$, $k \geq 0$, are distinct.
Then the sequence $\{ \cE (\tH^{(k)}, \alpha^{(k)}) \}_k$ of the energy values is decreasing. Consequently, it is convergent when $k$ tends to infinity.
\end{proposition}
\begin{proof}
Take any $k \geq 0$. Thanks to \cref{Prop:backtracking}, the values of $\tau_{\tH}$ and $\tau_{\alpha}$ at the $k$th iteration of \cref{Alg:proposed} are successfully determined in finite steps, say $\tau_{\tH}^{(k)}$ and $\tau_{\alpha}^{(k)}$, respectively. It follows that
\begin{subequations}
\begin{align}
\label{tH_decrease}
\cE (\tH^{(k)}, \alpha^{(k)}) &\geq \cE (\tH^{(k+1)}, \alpha^{(k)}) + \frac{1}{2 \tau_{\tH}^{(k)}} \| \tH^{(k+1)} - \tH^{(k)} \|_F^2, \\
\label{alpha_decrease}
\cE (\tH^{(k+1)}, \alpha^{(k)}) &\geq \cE (\tH^{(k+1)}, \alpha^{(k+1)}) + \frac{1}{2 \tau_{\alpha}^{(k)}} \| \alpha^{(k+1)} - \alpha^{(k)} \|_2^2.
\end{align}
\end{subequations}
Summing~\cref{tH_decrease,alpha_decrease} yields
\begin{equation*}
\cE (\tH^{(k)}, \alpha^{(k)}) \geq \cE (\tH^{(k+1)}, \alpha^{(k+1)}) + \frac{1}{2 \tau_{\tH}^{(k)}} \| \tH^{(k+1)} - \tH^{(k)} \|_F^2 + \frac{1}{2 \tau_{\alpha}^{(k)}} \| \alpha^{(k+1)} - \alpha^{(k)} \|_2^2,
\end{equation*}
which implies that the sequence $\{ \cE (\tH^{(k)}, \alpha^{(k)}) \}_k$ is decreasing. Moreover, since $\cE (\tH, \alpha)$ is bounded below~(see \cref{Lem:log}), we deduce that $\{ \cE (\tH^{(k)}, \alpha^{(k)}) \}_k$ is convergent when $t \rightarrow \infty$.
\end{proof}

\subsection{Initialization}
Since the proposed model~\cref{proposed_varpro} is nonconvex, an output of \cref{Alg:proposed} is an approximation for one of the local minima of the energy functional $\cE (\tH, \alpha)$. The corresponding local minimum is sensitive to the choice of an initial guess $(\tH^{(0)}, \alpha^{(0)})$ of \cref{Alg:proposed}~\cite{YL:2020}. Here, we deal with some ways to choose initial guesses of \cref{Alg:proposed} that yield satisfactory results. We note that the same issue was considered for the $\ell^2$-optimized DMD model in~\cite[section~3.2]{AK:2018}. Choosing $\tH^{(0)}$ is straightforward; we simply set $\tH^{(0)} = H$. Then we clearly have $\tH^{(0)} \in \cS_H^{\circ}$. Meanwhile, choosing $\alpha^{(0)}$ is relatively complicated.

The first simple strategy for choosing $\alpha^{(0)}$ is to adopt the initialization scheme~\cite[Algorithm~4]{AK:2018} designed for the $\ell^2$-optimized DMD model. The initialization routine~\cite[Algorithm~4]{AK:2018} solves a finite difference approximation of the target linear dynamical system~\cref{dynamic} using the given time series $X$. Since the $\ell^2$-optimized DMD model and the proposed model shares the same target linear dynamical system~\cref{dynamic}, we may adopt the initialization routine proposed by Askham and Kutz without modification.

The second strategy is to regard \cref{Alg:proposed} as a postprocessing step for the $\ell^2$-optimized DMD model. That is, we set $\alpha^{(0)}$ as the output of the $\ell^2$-optimized DMD model equipped with the same data $X$. Then, starting from the output of the $\ell^2$-optimized DMD model, \cref{Alg:proposed} will find an output that is more suitable for multiplicative noise according to the proposed model~\cref{proposed_varpro}.

It is not clear which of the two initialization strategies for $\alpha^{(0)}$ explained above results in better output than the other. To obtain a better result, we can do the following procedure: we run \cref{Alg:proposed} twice with the above-mentioned initialization schemes, then pick the one with smaller energy value $\cE(\tH, \alpha)$ among the outputs. This procedure is summarized in \cref{Alg:proposed_full}.

\begin{algorithm}
\caption{Full procedure to solve the proposed model~\cref{proposed_varpro} incorporating two initialization strategies}
\begin{algorithmic}[]
\label{Alg:proposed_full}
\STATE $\bullet$ Run \cref{Alg:proposed} with $(\tH^{(0)}, \alpha^{(0)}) = (H, \alpha^{0,1})$, where $\alpha^{0,1}$ is obtained by~\cite[Algorithm~4]{AK:2018}. Let $(\tH^{\star, 1}, \alpha^{\star,1})$ denote the result.
\STATE $\bullet$ Run \cref{Alg:proposed} with $(\tH^{(0)}, \alpha^{(0)}) = (H, \alpha^{0,2})$, where $\alpha^{0,2}$ is obtained by the $\ell^2$-optimized DMD model~\cref{AK_varpro}. Let $(\tH^{\star, 2}, \alpha^{\star,2})$ denote the result.
\IF {$\cE (\tH^{\star, 1}, \alpha^{\star,1}) \leq  \cE (\tH^{\star, 2}, \alpha^{\star,2})$}
\STATE \nonumber Return $(\tH^{\star, 1}, \alpha^{\star,1})$.
\ELSE
\STATE Return $(\tH^{\star, 2}, \alpha^{\star,2})$.
\ENDIF
\end{algorithmic}
\end{algorithm}

\section{Numerical experiments}
\label{Sec:Numerical}
In this section, we demonstrate the proposed DMD model to three numerical examples. Two of these examples are simple synthetic data presented in~\cite{AK:2018,DHWR:2016}, which represent a periodic linear system and a system containing hidden dynamics. The third example is the pressure fluctuation in a one-dimensional combustor~\cite{waugh2011}.

In the numerical experiments presented in this section, we used the stop criterion
\begin{equation*}
\frac{\| \alpha^{(k)} - \alpha^{(k-1)} \|_2 }{\| \alpha^{(k)} \|_2} < \mathrm{TOL},
\end{equation*}
for the Levenberg--Marquardt algorithm to solving the $\ell^2$-optimized DMD model~\cref{AK_varpro}, while we used
\begin{equation*}
\max \left\{ \frac{ \| \tH^{(k)} - \tH^{(k-1)}\|_F }{\| \tH^{(k)} \|_F}, \frac{\| \alpha^{(k)} - \alpha^{(k-1)} \|_2 }{\| \alpha^{(k)} \|_2}  \right\} < \mathrm{TOL},
\end{equation*}
for~\cref{Alg:proposed} solving the proposed model~\cref{proposed_varpro}, where $\mathrm{TOL} = 10^{-5}$. We often use the following distance function allowing permutations over entries to measure how far two vectors are apart from each other:
\begin{equation}
\label{d}
d(\alpha_1, \alpha_2) = \min \left\{ \| \tilde{\alpha}_1 - \tilde{\alpha}_2 \|_2
: \tilde{\alpha}_1 \text{ and } \tilde{\alpha}_2 \text{ are permutations of } \alpha_1 \text{ and } \alpha_2 \text{, respectively.}\right\}
\end{equation}
We also measure the reconstruction error $E_{\recon}$ defined as follows for the sake of comparison among various DMD models:
\begin{equation}
    \label{recon}
    E_{\recon} = \frac{\|H_{\clean}-H_{\recon} \|_F}{\|H_{\clean}\|_F},
\end{equation}
where $H_{\clean}$ is the snapshots of the zero-noise data and $H_{\recon}$ is given by
\begin{equation*}
H_{\recon} = \begin{cases}
\Phi (\alpha) \Phi (\alpha)^{\dag} H & \textrm{ for the $\ell^2$-optimized DMD model~\cref{AK_varpro}}, \\
\Phi (\alpha) \Phi (\alpha)^{\dag} \tH & \textrm{ for the proposed model~\cref{proposed_varpro}}.
\end{cases}
\end{equation*}
For the step sizes of \cref{Alg:proposed}, we simply set $\tau_{\tH} = 0.1$ and $\tau_{\alpha} = 0.1/\eta$; thanks to the backtracking processes in \cref{Alg:proposed}, initial choices on $\tau_{\tH}$ and $\tau_{\alpha}$ do not critically affect on the performance of the proposed model. All algorithms presented in this section are implemented in MATLAB R2020b and are performed on a computer equipped with two Intel Xeon SP-6148 CPUs~(2.4GHz, 20C), 384GB RAM, and the operating system CentOS 7.8 64-bit.

\subsection{Periodic system}
As the first example, we consider the two-dimensional linear system
\begin{equation*}
\dot{z} = \begin{bmatrix} 1 & -2 \\ 1 & -1\end{bmatrix} z
\end{equation*}
with the initial condition $z (0) = [1, 0.1]^{\rT}$. It is straightforward to check that the solution of the above system is given by
\begin{equation}
\label{periodic_clean}
z (t) = \begin{bmatrix} z_1 (t) \\ z_2(t) \end{bmatrix}
= \begin{bmatrix} \sin t + \cos t \\ \sin t \end{bmatrix}
+ 0.1 \begin{bmatrix} -2 \sin t \\ \cos t - \sin t\end{bmatrix}.
\end{equation}
The continuous-time eigenvalues of~\cref{periodic_clean} are $\pm i$. Snapshots $\{ \xi^n \}_n$ are taken as
\begin{equation}
\label{periodic_noisy}
\xi^n = \begin{bmatrix} z_1((n-1) \Delta t) \epsilon_{1n} \\ z_2 ((n-1) \Delta t) \epsilon_{2n} \end{bmatrix}, \quad 1 \leq n \leq N,
\end{equation}
where $\Delta t = 0.1$ and $\epsilon_{mn}$~($m= 1,2$) represents multiplicative noise following the gamma distribution of mean $1$ and variance $\sigma^2$. 

\begin{table}
\begin{subtable}[]{\textwidth}
\centering
\begin{adjustbox}{max width=\textwidth}
\begin{tabular}{c | c c c c c} \hline
\multirow{2}{*}{$N$} & \multirow{2}{*}{$\ell^2$-optimized DMD} & \multicolumn{4}{c}{Proposed model} \\
\cline{3-6}
& & $\eta = 10^1$ & $\eta = 10^2$ & $\eta = 10^3$ & $\eta = 10^4$ \\
\hline
$2^4$
& \begin{tabular}{c}$9.14 \times 10^{-1}$ \\ ($7.02 \times 10^{-1}$) \end{tabular} 
& \begin{tabular}{c}$7.33 \times 10^{-1}$ \\ ($6.40 \times 10^{-1}$) \end{tabular} 
& \begin{tabular}{c}$6.68 \times 10^{-1}$ \\ ($4.81 \times 10^{-1}$) \end{tabular} 
& \begin{tabular}{c}$6.03 \times 10^{-1}$ \\ ($4.15 \times 10^{-1}$) \end{tabular} 
& \begin{tabular}{c}$4.43 \times 10^{-1}$ \\ ($3.86 \times 10^{-1}$) \end{tabular} \\ 
$2^5$
& \begin{tabular}{c}$1.07 \times 10^{-1}$ \\ ($6.37 \times 10^{-2}$) \end{tabular} 
& \begin{tabular}{c}$7.23 \times 10^{-2}$ \\ ($4.79 \times 10^{-2}$) \end{tabular} 
& \begin{tabular}{c}$6.15 \times 10^{-2}$ \\ ($4.32 \times 10^{-2}$) \end{tabular} 
& \begin{tabular}{c}$5.46 \times 10^{-2}$ \\ ($4.10 \times 10^{-2}$) \end{tabular} 
& \begin{tabular}{c}$3.75 \times 10^{-2}$ \\ ($2.92 \times 10^{-2}$) \end{tabular} \\
$2^6$
& \begin{tabular}{c}$3.00 \times 10^{-2}$ \\ ($1.65 \times 10^{-2}$) \end{tabular} 
& \begin{tabular}{c}$1.85 \times 10^{-2}$ \\ ($1.18 \times 10^{-2}$) \end{tabular} 
& \begin{tabular}{c}$1.44 \times 10^{-2}$ \\ ($1.10 \times 10^{-2}$) \end{tabular} 
& \begin{tabular}{c}$1.20 \times 10^{-2}$ \\ ($1.02 \times 10^{-2}$) \end{tabular} 
& \begin{tabular}{c}$1.28 \times 10^{-2}$ \\ ($9.66 \times 10^{-3}$) \end{tabular} \\
$2^7$
& \begin{tabular}{c}$1.00 \times 10^{-2}$ \\ ($5.36 \times 10^{-3}$) \end{tabular} 
& \begin{tabular}{c}$5.71 \times 10^{-3}$ \\ ($3.48 \times 10^{-3}$) \end{tabular} 
& \begin{tabular}{c}$4.21 \times 10^{-3}$ \\ ($3.21 \times 10^{-3}$) \end{tabular} 
& \begin{tabular}{c}$3.34 \times 10^{-3}$ \\ ($3.16 \times 10^{-3}$) \end{tabular} 
& \begin{tabular}{c}$4.61 \times 10^{-3}$ \\ ($3.32 \times 10^{-3}$) \end{tabular} \\
\hline
\end{tabular}
\end{adjustbox}
\caption{$\sigma^2 = 10^{-1}$}
\end{subtable}

\begin{subtable}[]{\textwidth}
\centering
\begin{adjustbox}{max width=\textwidth}
\begin{tabular}{c | c c c c c} \hline
\multirow{2}{*}{$N$} & \multirow{2}{*}{$\ell^2$-optimized DMD} & \multicolumn{4}{c}{Proposed model} \\
\cline{3-6}
& & $\eta = 10^1$ & $\eta = 10^2$ & $\eta = 10^3$ & $\eta = 10^4$ \\
\hline
$2^4$
& \begin{tabular}{c}$2.41 \times 10^{-1}$ \\ ($1.52 \times 10^{-1}$) \end{tabular} 
& \begin{tabular}{c}$1.94 \times 10^{-1}$ \\ ($1.19 \times 10^{-1}$) \end{tabular} 
& \begin{tabular}{c}$1.88 \times 10^{-1}$ \\ ($1.15 \times 10^{-1}$) \end{tabular} 
& \begin{tabular}{c}$1.73 \times 10^{-1}$ \\ ($1.17 \times 10^{-1}$) \end{tabular} 
& \begin{tabular}{c}$8.31 \times 10^{-2}$ \\ ($5.83 \times 10^{-2}$) \end{tabular} \\ 
$2^5$
& \begin{tabular}{c}$3.34 \times 10^{-2}$ \\ ($1.90 \times 10^{-2}$) \end{tabular} 
& \begin{tabular}{c}$1.94 \times 10^{-2}$ \\ ($1.32 \times 10^{-2}$) \end{tabular} 
& \begin{tabular}{c}$1.57 \times 10^{-2}$ \\ ($1.13 \times 10^{-2}$) \end{tabular} 
& \begin{tabular}{c}$1.29 \times 10^{-2}$ \\ ($1.06 \times 10^{-2}$) \end{tabular} 
& \begin{tabular}{c}$9.28 \times 10^{-3}$ \\ ($5.66 \times 10^{-3}$) \end{tabular} \\
$2^6$
& \begin{tabular}{c}$9.45 \times 10^{-3}$ \\ ($5.34 \times 10^{-3}$) \end{tabular} 
& \begin{tabular}{c}$4.61 \times 10^{-3}$ \\ ($2.64 \times 10^{-3}$) \end{tabular} 
& \begin{tabular}{c}$2.87 \times 10^{-3}$ \\ ($1.93 \times 10^{-3}$) \end{tabular} 
& \begin{tabular}{c}$1.61 \times 10^{-3}$ \\ ($1.41 \times 10^{-2}$) \end{tabular} 
& \begin{tabular}{c}$3.44 \times 10^{-3}$ \\ ($2.25 \times 10^{-3}$) \end{tabular} \\
$2^7$
& \begin{tabular}{c}$3.25 \times 10^{-3}$ \\ ($1.82 \times 10^{-3}$) \end{tabular} 
& \begin{tabular}{c}$1.52 \times 10^{-3}$ \\ ($8.68 \times 10^{-4}$) \end{tabular} 
& \begin{tabular}{c}$8.24 \times 10^{-4}$ \\ ($4.79 \times 10^{-3}$) \end{tabular} 
& \begin{tabular}{c}$3.61 \times 10^{-4}$ \\ ($3.17 \times 10^{-4}$) \end{tabular} 
& \begin{tabular}{c}$1.18 \times 10^{-3}$ \\ ($7.54 \times 10^{-4}$) \end{tabular} \\
\hline
\end{tabular}
\end{adjustbox}
\caption{$\sigma^2 = 10^{-2}$}
\end{subtable}

\begin{subtable}[]{\textwidth}
\centering
\begin{adjustbox}{max width=\textwidth}
\begin{tabular}{c | c c c c c} \hline
\multirow{2}{*}{$N$} & \multirow{2}{*}{$\ell^2$-optimized DMD} & \multicolumn{4}{c}{Proposed model} \\
\cline{3-6}
& & $\eta = 10^1$ & $\eta = 10^2$ & $\eta = 10^3$ & $\eta = 10^4$ \\
\hline
$2^4$
& \begin{tabular}{c}$7.45 \times 10^{-2}$ \\ ($4.35 \times 10^{-2}$) \end{tabular} 
& \begin{tabular}{c}$6.08 \times 10^{-2}$ \\ ($3.72 \times 10^{-2}$) \end{tabular} 
& \begin{tabular}{c}$5.73 \times 10^{-2}$ \\ ($3.52 \times 10^{-2}$) \end{tabular} 
& \begin{tabular}{c}$4.21 \times 10^{-2}$ \\ ($3.15 \times 10^{-2}$) \end{tabular} 
& \begin{tabular}{c}$2.15 \times 10^{-2}$ \\ ($1.19 \times 10^{-2}$) \end{tabular} \\ 
$2^5$
& \begin{tabular}{c}$1.03 \times 10^{-2}$ \\ ($6.20 \times 10^{-3}$) \end{tabular} 
& \begin{tabular}{c}$5.35 \times 10^{-3}$ \\ ($3.28 \times 10^{-3}$) \end{tabular} 
& \begin{tabular}{c}$3.69 \times 10^{-3}$ \\ ($2.76 \times 10^{-3}$) \end{tabular} 
& \begin{tabular}{c}$2.19 \times 10^{-3}$ \\ ($1.65 \times 10^{-3}$) \end{tabular} 
& \begin{tabular}{c}$2.95 \times 10^{-3}$ \\ ($1.83 \times 10^{-3}$) \end{tabular} \\
$2^6$
& \begin{tabular}{c}$3.05 \times 10^{-3}$ \\ ($1.68 \times 10^{-3}$) \end{tabular} 
& \begin{tabular}{c}$1.32 \times 10^{-3}$ \\ ($7.43 \times 10^{-4}$) \end{tabular} 
& \begin{tabular}{c}$7.56 \times 10^{-4}$ \\ ($4.34 \times 10^{-4}$) \end{tabular} 
& \begin{tabular}{c}$3.44 \times 10^{-4}$ \\ ($2.17 \times 10^{-4}$) \end{tabular} 
& \begin{tabular}{c}$1.15 \times 10^{-3}$ \\ ($6.45 \times 10^{-4}$) \end{tabular} \\
$2^7$
& \begin{tabular}{c}$1.04 \times 10^{-3}$ \\ ($5.90 \times 10^{-4}$) \end{tabular} 
& \begin{tabular}{c}$4.15 \times 10^{-4}$ \\ ($2.33 \times 10^{-4}$) \end{tabular} 
& \begin{tabular}{c}$2.29 \times 10^{-4}$ \\ ($1.35 \times 10^{-4}$) \end{tabular} 
& \begin{tabular}{c}$6.23 \times 10^{-5}$ \\ ($4.41 \times 10^{-5}$) \end{tabular} 
& \begin{tabular}{c}$7.87 \times 10^{-5}$ \\ ($8.93 \times 10^{-5}$) \end{tabular} \\
\hline
\end{tabular}
\end{adjustbox}
\caption{$\sigma^2 = 10^{-3}$}
\end{subtable}
\caption{The average and sample standard deviation~(given in parentheses) of the eigenvalue error $d(\alpha^{\star}, \alpha_{\exact})$ over 1,000 independent trials for the problem~\cref{periodic_noisy} when the noise variance $\sigma^2$, the number of snapshots $N$, and the penalty parameter $\eta$ vary, where $d(\cdot, \cdot)$ is given in~\cref{d}, $\alpha_{\exact} = [i, -i]^{\rT}$ denotes the exact eigenvalue of the clean dynamics~\cref{periodic_clean}, and $\alpha^{\star}$ is a solution obtained by either the $\ell^2$-optimized DMD model~\cref{AK_varpro} or the proposed model~\cref{proposed_varpro} with the initial guess $\alpha^{(0)} = \alpha_{\exact}$.}
\label{Table:periodic_valid}
\end{table}

To validate that the proposed model is well-designed for multiplicative noise, we verify whether there exists a local minimum of~\cref{proposed_varpro} very close to the exact eigenvalues $\alpha_{\exact} = [i, -i]^{\rT}$ when $X = H^{\rT}$ is given by~\cref{periodic_noisy}. \Cref{Table:periodic_valid} presents the average and sample standard deviation of $d(\alpha^{\star}, \alpha_{\exact})$ over 1,000 independent trials for various settings on the noise variance $\sigma^2$, the number of snapshots $N$, and the penalty parameter $\eta$, where $\alpha^{\star}$ is a solution obtained by either the $\ell^2$-optimized DMD model~\cref{AK_varpro} or the proposed model~\cref{proposed_varpro} with the initial guess $\alpha^{(0)} = \alpha_{\exact}$. In all cases, the proposed model results in smaller $d(\alpha^{\star}, \alpha_{\exact})$ than that of the $\ell^2$-optimized DMD model. This means that~\cref{proposed_varpro} possesses a local minimum close to $\alpha_{\exact}$, and the distance between the local minimum and $\alpha_{\exact}$ is smaller than that of~\cref{AK_varpro}. We can also observe that the proposed model performs especially well when the value of $\eta$ is around $10^3$ and $10^4$, i.e., when $\eta$ is sufficiently large. This verifies \cref{Prop:gamma}; when $\eta$ is large enough, the proposed model~\cref{proposed_varpro} behaves like~\cref{proposed_nonpenalty}, which is an optimal model for multiplicative noise in the sense of MAP estimator, so that it gives more accurate results than other models when the data is corrupted by multiplicative noise.

\begin{figure}[]
\centering
\includegraphics[width=0.95\linewidth]{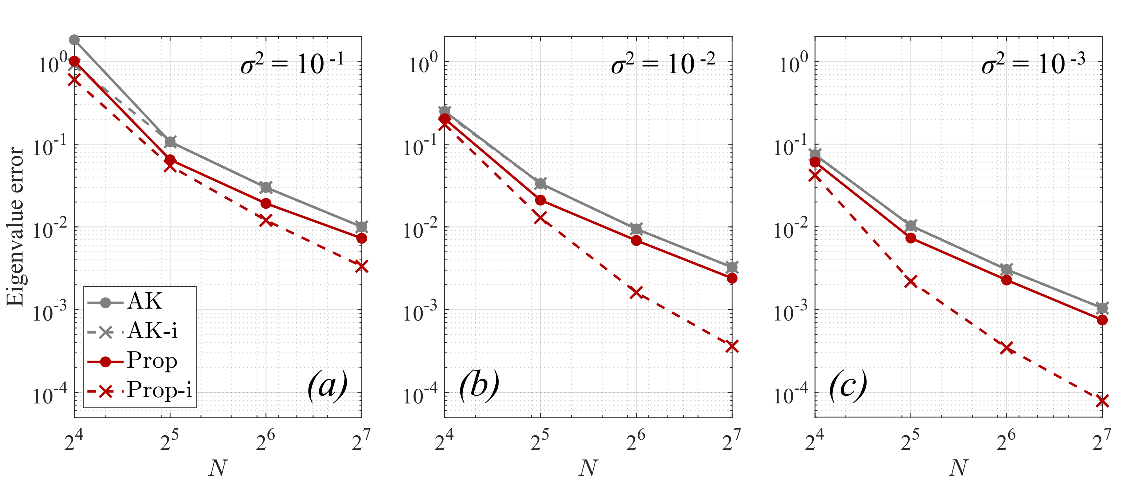}
\caption{The average of the eigenvalue error $d(\alpha^{\star}, \alpha_{\exact})$ over 1,000 independent trials for the problem~\cref{periodic_noisy} with respect to various numbers of snapshots $N$, where $d(\cdot, \cdot)$ is given in~\cref{d} and $\alpha_{\exact} = [i, -i]^{\rT}$. Here, $\alpha^{\star}$ denotes a solution obtained by either the $\ell^2$-optimized DMD model~\cref{AK_varpro}~(AK) or the proposed model~\cref{proposed_varpro}~(Prop) with $\eta = 10^3$.
The expression ``-i'' means that an ideal initial guess $\alpha^{(0)} = \alpha_{\exact}$ is used.}
\label{Fig:periodic_error}
\end{figure}

\begin{figure}[]
\centering
\includegraphics[width=0.95\linewidth]{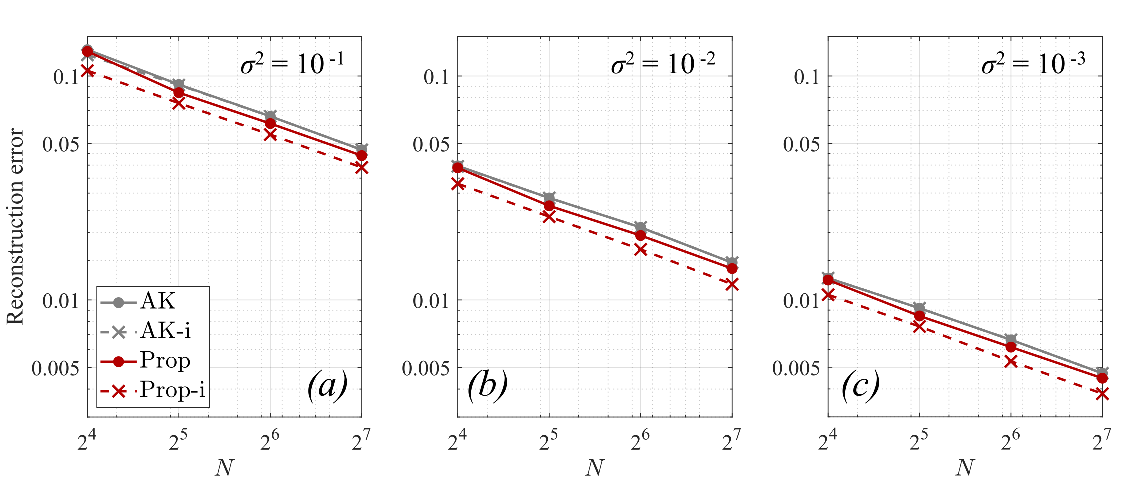}
\caption{The average of the reconstruction error $E_{\recon}$ over 1,000 independent trials for the problem~\cref{periodic_noisy} with respect to various numbers of snapshots $N$, where $E_{\recon}$ is given in~\cref{recon} and $\alpha_{\exact} = [i, -i]^{\rT}$. Here, $\alpha^{\star}$ denotes a solution obtained by either the $\ell^2$-optimized DMD model~\cref{AK_varpro}~(AK) or the proposed model~\cref{proposed_varpro}~(Prop) with $\eta = 10^3$.
The expression ``-i'' means that an ideal initial guess $\alpha^{(0)} = \alpha_{\exact}$ is used.}
\label{Fig:periodic_recon}
\end{figure}

Next, we observe the practical performance of the proposed model. We set $\eta = 10^3$ in~\cref{proposed_varpro}. The following algorithms are used in the numerical experiments for \cref{Fig:periodic_error,Fig:periodic_recon}:
\begin{itemize}
\item AK: Levenberg--Marquardt algorithm to solve the $\ell^2$-optimized DMD model~\cref{AK_varpro} with the initial guess determined by~\cite[Algorithm~4]{AK:2018}.
\item AK-i: Levenberg--Marquardt algorithm to solve the $\ell^2$-optimized DMD model~\cref{AK_varpro} with an ideal initial guess $\alpha^{(0)} = \alpha_{\exact}$.
\item Prop: \cref{Alg:proposed_full} to solve the proposed model~\cref{proposed_varpro}.
\item Prop-i: \cref{Alg:proposed} to solve the proposed model~\cref{proposed_varpro} with an ideal initial guess $\alpha^{(0)} = \alpha_{\exact}$.
\end{itemize}
Algorithms AK and Prop give practical performances of the $\ell^2$-optimized DMD and proposed models, while AK-i and Prop-i are intended to show ideal performances of two models, respectively. That is, AK-i and Prop-i represent a scenario that we can find a good initial guess $\alpha^{(0)}$ that is almost identical to $\alpha_{\exact}$. However, we do not have such an ideal situation in most cases, and what we can generally expect from the $\ell^2$-optimized DMD and proposed models are the results of AK and Prop. \Cref{Fig:periodic_error} shows the average of $d(\alpha^{\star}, \alpha_{\exact})$ over 1,000 independent trials for the problem~\cref{periodic_noisy} with respect to varying $N$, where $\alpha^{\star}$ denotes a solution obtained by either of the above four algorithms. \Cref{Fig:periodic_recon} presents $E_{\recon}$ defined in~\cref{recon} under the same setting as \cref{Fig:periodic_error}. All algorithms perform well in the sense that both the eigenvalue and reconstruction errors decay as the number of snapshots $N$ increases. Across all values of $\sigma^2$ and $N$, AK and AK-i result in almost identical eigenvalue and reconstruction errors, while Prop showed lower eigenvalue and reconstruction errors than them. From this observation, we can conclude that the proposed model is more effective than the $\ell^2$-optimized DMD model, even in a situation that an ideal initial guess for the $\ell^2$-optimized DMD model is available. Another interesting aspect of \cref{Fig:periodic_error} is that the error curves of Prop-i lie much lower than those of all the other algorithms. Although Prop showed better performance than AK-i and AK, there are additional gaps between the errors of Prop-i and Prop. This demonstrates that finding better initialization schemes for \cref{Alg:proposed} can significantly improve the performance of the proposed model; designing better initialization schemes remains a topic for the future research.

\subsection{Hidden dynamics}
In~\cite{AK:2018,DHWR:2016}, it was mentioned that identifying the dynamics of a signal containing rapidly decaying components is quite challenging. We consider the signal
\begin{equation}
\label{hidden_clean}
z(x,t) = \sin (x - t) e^{t} + \sin (0.4 x - 3.7 t) e^{-0.2 t},
\end{equation}
which is a superposition of two travelling sinusoidal signals, with one growing and the other decaying~\cite{AK:2018,DHWR:2016}. It has four continuous-time eigenvalues $1 \pm i$ and $-0.2 \pm 3.7 i$.
We set the spatial domain as $[0, 15]$ and use 300 equispaced points to discretize, i.e., $M = 300$ and $\Delta x = 15/(M-1)$. We also set the temporal domain as $[0,1]$, with $N$ equispaced discretization points, i.e., $\Delta t = 1/(N-1)$. In this setting, the corresponding matrix of snapshots $X$ corrupted by gamma multiplicative noise is given by
\begin{equation}
\label{hidden_noisy}
X_{mn} = z((m-1) \Delta x, (n-1) \Delta t) \epsilon_{mn}, \quad 1 \leq m \leq M, \gap 1 \leq n \leq N,
\end{equation}
where $\epsilon_{mn}$ follows the gamma distribution of mean $1$ and variance $\sigma^2$.

\begin{table}
\begin{subtable}[]{\textwidth}
\centering
\begin{adjustbox}{max width=\textwidth}
\begin{tabular}{c | c c c c c} \hline
\multirow{2}{*}{$N$} & \multirow{2}{*}{$\ell^2$-optimized DMD} & \multicolumn{4}{c}{Proposed model} \\
\cline{3-6}
& & $\eta = 10^3$ & $\eta = 10^4$ & $\eta = 10^5$ & $\eta = 10^6$ \\
\hline
$2^4$
& \begin{tabular}{c}$3.77 \times 10^{0}$ \\ ($3.64 \times 10^{0}$) \end{tabular} 
& \begin{tabular}{c}$3.06 \times 10^{0}$ \\ ($2.13 \times 10^{0}$) \end{tabular} 
& \begin{tabular}{c}$2.94 \times 10^{0}$ \\ ($1.01 \times 10^{0}$) \end{tabular} 
& \begin{tabular}{c}$2.76 \times 10^{0}$ \\ ($6.75 \times 10^{-1}$) \end{tabular} 
& \begin{tabular}{c}$2.90 \times 10^{0}$ \\ ($8.65 \times 10^{-1}$) \end{tabular} \\ 
$2^5$
& \begin{tabular}{c}$2.38 \times 10^{0}$ \\ ($9.44 \times 10^{-1}$) \end{tabular} 
& \begin{tabular}{c}$2.04 \times 10^{0}$ \\ ($4.48 \times 10^{-1}$) \end{tabular} 
& \begin{tabular}{c}$1.83 \times 10^{0}$ \\ ($4.15 \times 10^{-1}$) \end{tabular} 
& \begin{tabular}{c}$1.97 \times 10^{0}$ \\ ($4.45 \times 10^{-1}$) \end{tabular} 
& \begin{tabular}{c}$1.94 \times 10^{0}$ \\ ($2.56 \times 10^{-1}$) \end{tabular} \\
$2^6$
& \begin{tabular}{c}$1.37 \times 10^{0}$ \\ ($5.35 \times 10^{-1}$) \end{tabular} 
& \begin{tabular}{c}$1.28 \times 10^{0}$ \\ ($4.66 \times 10^{-1}$) \end{tabular} 
& \begin{tabular}{c}$1.14 \times 10^{0}$ \\ ($3.93 \times 10^{-1}$) \end{tabular} 
& \begin{tabular}{c}$1.13 \times 10^{0}$ \\ ($3.91 \times 10^{-1}$) \end{tabular} 
& \begin{tabular}{c}$1.20 \times 10^{0}$ \\ ($4.20 \times 10^{-1}$) \end{tabular} \\
$2^7$
& \begin{tabular}{c}$7.26 \times 10^{-1}$ \\ ($3.22 \times 10^{-1}$) \end{tabular} 
& \begin{tabular}{c}$7.00 \times 10^{-1}$ \\ ($2.97 \times 10^{-1}$) \end{tabular} 
& \begin{tabular}{c}$6.85 \times 10^{-1}$ \\ ($2.90 \times 10^{-1}$) \end{tabular} 
& \begin{tabular}{c}$6.80 \times 10^{-1}$ \\ ($2.73 \times 10^{-1}$) \end{tabular} 
& \begin{tabular}{c}$6.64 \times 10^{-1}$ \\ ($2.62 \times 10^{-1}$) \end{tabular} \\
\hline
\end{tabular}
\end{adjustbox}
\caption{$\sigma^2 = 2^{-5}$}
\end{subtable}

\begin{subtable}[]{\textwidth}
\centering
\begin{adjustbox}{max width=\textwidth}
\begin{tabular}{c | c c c c c} \hline
\multirow{2}{*}{$N$} & \multirow{2}{*}{$\ell^2$-optimized DMD} & \multicolumn{4}{c}{Proposed model} \\
\cline{3-6}
& & $\eta = 10^3$ & $\eta = 10^4$ & $\eta = 10^5$ & $\eta = 10^6$ \\
\hline
$2^4$
& \begin{tabular}{c}$1.04 \times 10^{0}$ \\ ($4.36 \times 10^{0}$) \end{tabular} 
& \begin{tabular}{c}$1.01 \times 10^{0}$ \\ ($4.04 \times 10^{0}$) \end{tabular} 
& \begin{tabular}{c}$1.02 \times 10^{0}$ \\ ($4.13 \times 10^{0}$) \end{tabular} 
& \begin{tabular}{c}$9.76 \times 10^{-1}$ \\ ($3.87 \times 10^{-1}$) \end{tabular} 
& \begin{tabular}{c}$9.59 \times 10^{-1}$ \\ ($3.60 \times 10^{-1}$) \end{tabular} \\ 
$2^5$
& \begin{tabular}{c}$6.12 \times 10^{-1}$ \\ ($2.75 \times 10^{-1}$) \end{tabular} 
& \begin{tabular}{c}$5.72 \times 10^{-1}$ \\ ($2.36 \times 10^{-1}$) \end{tabular} 
& \begin{tabular}{c}$5.41 \times 10^{-1}$ \\ ($2.17 \times 10^{-1}$) \end{tabular} 
& \begin{tabular}{c}$5.18 \times 10^{-1}$ \\ ($1.88 \times 10^{-1}$) \end{tabular} 
& \begin{tabular}{c}$5.33 \times 10^{-1}$ \\ ($2.06 \times 10^{-1}$) \end{tabular} \\
$2^6$
& \begin{tabular}{c}$3.72 \times 10^{-1}$ \\ ($1.62 \times 10^{-1}$) \end{tabular} 
& \begin{tabular}{c}$3.50 \times 10^{-1}$ \\ ($1.50 \times 10^{-1}$) \end{tabular} 
& \begin{tabular}{c}$3.32 \times 10^{-1}$ \\ ($1.27 \times 10^{-1}$) \end{tabular} 
& \begin{tabular}{c}$3.16 \times 10^{-1}$ \\ ($1.12 \times 10^{-1}$) \end{tabular} 
& \begin{tabular}{c}$3.35 \times 10^{-1}$ \\ ($1.34 \times 10^{-1}$) \end{tabular} \\
$2^7$
& \begin{tabular}{c}$2.36 \times 10^{-1}$ \\ ($9.45 \times 10^{-2}$) \end{tabular} 
& \begin{tabular}{c}$2.24 \times 10^{-1}$ \\ ($9.50 \times 10^{-2}$) \end{tabular} 
& \begin{tabular}{c}$2.30 \times 10^{-1}$ \\ ($9.24 \times 10^{-2}$) \end{tabular} 
& \begin{tabular}{c}$2.17 \times 10^{-1}$ \\ ($8.23 \times 10^{-2}$) \end{tabular} 
& \begin{tabular}{c}$2.24 \times 10^{-1}$ \\ ($8.90 \times 10^{-2}$) \end{tabular} \\
\hline
\end{tabular}
\end{adjustbox}
\caption{$\sigma^2 = 2^{-7}$}
\end{subtable}

\begin{subtable}[]{\textwidth}
\centering
\begin{adjustbox}{max width=\textwidth}
\begin{tabular}{c | c c c c c} \hline
\multirow{2}{*}{$N$} & \multirow{2}{*}{$\ell^2$-optimized DMD} & \multicolumn{4}{c}{Proposed model} \\
\cline{3-6}
& & $\eta = 10^3$ & $\eta = 10^4$ & $\eta = 10^5$ & $\eta = 10^6$ \\
\hline
$2^4$
& \begin{tabular}{c}$2.99 \times 10^{-1}$ \\ ($1.26 \times 10^{-1}$) \end{tabular} 
& \begin{tabular}{c}$2.94 \times 10^{-1}$ \\ ($1.24 \times 10^{-1}$) \end{tabular} 
& \begin{tabular}{c}$2.93 \times 10^{-1}$ \\ ($1.20 \times 10^{-1}$) \end{tabular} 
& \begin{tabular}{c}$2.89 \times 10^{-1}$ \\ ($1.11 \times 10^{-1}$) \end{tabular} 
& \begin{tabular}{c}$2.86 \times 10^{-1}$ \\ ($1.12 \times 10^{-1}$) \end{tabular} \\ 
$2^5$
& \begin{tabular}{c}$2.13 \times 10^{-1}$ \\ ($8.50 \times 10^{-2}$) \end{tabular} 
& \begin{tabular}{c}$1.88 \times 10^{-1}$ \\ ($7.90 \times 10^{-2}$) \end{tabular} 
& \begin{tabular}{c}$1.86 \times 10^{-1}$ \\ ($7.50 \times 10^{-2}$) \end{tabular} 
& \begin{tabular}{c}$1.82 \times 10^{-1}$ \\ ($7.10 \times 10^{-2}$) \end{tabular} 
& \begin{tabular}{c}$1.82 \times 10^{-1}$ \\ ($7.54 \times 10^{-2}$) \end{tabular} \\
$2^6$
& \begin{tabular}{c}$1.49 \times 10^{-1}$ \\ ($5.90 \times 10^{-2}$) \end{tabular} 
& \begin{tabular}{c}$1.30 \times 10^{-1}$ \\ ($5.46 \times 10^{-2}$) \end{tabular} 
& \begin{tabular}{c}$1.28 \times 10^{-1}$ \\ ($5.43 \times 10^{-2}$) \end{tabular} 
& \begin{tabular}{c}$1.24 \times 10^{-1}$ \\ ($5.07 \times 10^{-2}$) \end{tabular} 
& \begin{tabular}{c}$1.25 \times 10^{-1}$ \\ ($5.13 \times 10^{-2}$) \end{tabular} \\
$2^7$
& \begin{tabular}{c}$1.04 \times 10^{-1}$ \\ ($4.40 \times 10^{-2}$) \end{tabular} 
& \begin{tabular}{c}$9.68 \times 10^{-2}$ \\ ($4.01 \times 10^{-2}$) \end{tabular} 
& \begin{tabular}{c}$9.81 \times 10^{-2}$ \\ ($4.19 \times 10^{-2}$) \end{tabular} 
& \begin{tabular}{c}$9.58 \times 10^{-2}$ \\ ($4.03 \times 10^{-2}$) \end{tabular} 
& \begin{tabular}{c}$9.28 \times 10^{-2}$ \\ ($4.07 \times 10^{-2}$) \end{tabular} \\
\hline
\end{tabular}
\end{adjustbox}
\caption{$\sigma^2 = 2^{-9}$}
\end{subtable}
\caption{The average and sample standard deviation~(given in parentheses) of the eigenvalue error $d(\alpha^{\star}, \alpha_{\exact})$ over 1,000 independent trials for the problem~\cref{hidden_noisy} when the noise variance $\sigma^2$, the number of snapshots $N$, and the penalty parameter $\eta$ vary, where $d(\cdot, \cdot)$ is given in~\cref{d}, $\alpha_{\exact} = [1 + i, 1 - i, -0.2 + 3.7i, -0.2 - 3.7i]^{\rT}$ denotes the exact eigenvalue of the clean dynamics~\cref{hidden_clean}, and $\alpha^{\star}$ is a solution obtained by either the $\ell^2$-optimized DMD model~\cref{AK_varpro} or the proposed model~\cref{proposed_varpro} with the initial guess $\alpha^{(0)} = \alpha_{\exact}$.}
\label{Table:hidden_valid}
\end{table}

As for the case of the periodic problem, we conduct numerical experiments that can validate the suitability of the proposed model to the multiplicative noise. The average and sample standard deviation of the eigenvalue errors $d(\alpha^{\star}, \alpha_{\exact})$ over 1,000 independent trials for various $\sigma^2$, $N$, and $\eta$ are presented in \cref{Table:hidden_valid}, where $\alpha^{\star}$ denotes a solution obtained by either~\cref{AK_varpro} or~\cref{proposed_varpro} with the ideal initial guess $\alpha^{(0)} = \alpha_{\exact}$. We can again assure that the eigenvalue errors resulted from the proposed model are smaller than that of the $\ell^2$-optimized DMD model for all cases. Hence, the same discussion as in \cref{Table:periodic_valid} can be made for the hidden dynamics problem. In \cref{Table:hidden_valid}, an optimal range for the penalty parameter $\eta$ that results in the best performance of the proposed model is observed to be between $10^5$ and $10^6$.

\begin{figure}[]
\centering
\includegraphics[width=0.95\linewidth]{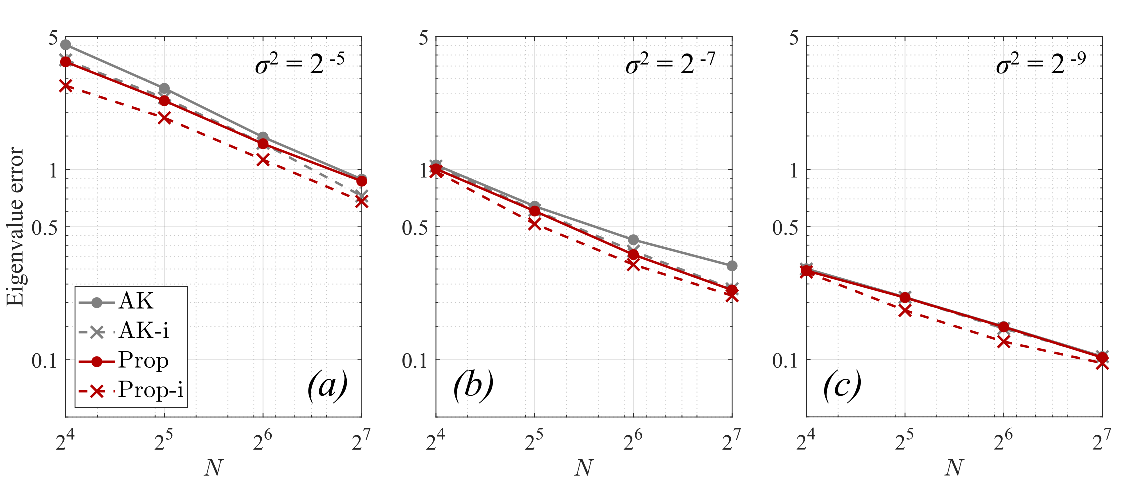}
\caption{The average of the eigenvalue error $d(\alpha^{\star}, \alpha_{\exact})$ over 1,000 independent trials for the problem~\cref{hidden_noisy} with respect to various numbers of snapshots $N$, where $d(\cdot, \cdot)$ is given in~\cref{d} and $\alpha_{\exact} = [1 + i, 1 - i, -0.2 + 3.7i, -0.2 - 3.7i]^{\rT}$. Here, $\alpha^{\star}$ denotes a solution obtained by either the $\ell^2$-optimized DMD model~\cref{AK_varpro}~(AK) or the proposed model~\cref{proposed_varpro}~(Prop) with $\eta = 10^5$. The expression ``-i'' means that an ideal initial guess $\alpha^{(0)} = \alpha_{\exact}$ is used.}
\label{Fig:hidden_error}
\end{figure}

\begin{figure}[]
\centering
\includegraphics[width=0.95\linewidth]{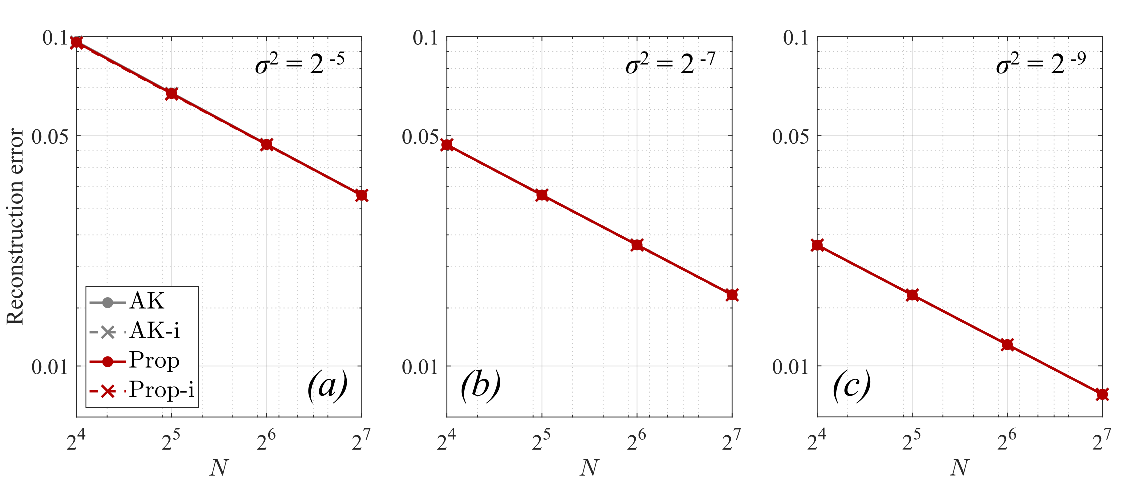}
\caption{The average of the reconstruction error $E_{\recon}$ over 1,000 independent trials for the problem~\cref{hidden_noisy} with respect to various numbers of snapshots $N$, where $E_{\recon}$ is given in~\cref{recon} and $\alpha_{\exact} = [1 + i, 1 - i, -0.2 + 3.7i, -0.2 - 3.7i]^{\rT}$. Here, $\alpha^{\star}$ denotes a solution obtained by either the $\ell^2$-optimized DMD model~\cref{AK_varpro}~(AK) or the proposed model~\cref{proposed_varpro}~(Prop) with $\eta = 10^5$. The expression ``-i'' means that an ideal initial guess $\alpha^{(0)} = \alpha_{\exact}$ is used.}
\label{Fig:hidden_recon}
\end{figure}

To evaluate the practical performance of the proposed model for the hidden dynamics problem, we compare four algorithms AK, AK-i, Prop, and Prop-i for solving~\cref{hidden_noisy}. In all experiments for \cref{Fig:hidden_error,Fig:hidden_recon}, we set $\eta = 10^5$ in~\cref{proposed_varpro}. \Cref{Fig:hidden_error} presents the average of the eigenvalue errors $d(\alpha^{\star}, \alpha_{\exact})$ over 1,000 trials for the problem~\cref{hidden_noisy} with respect to varying $N$. In all algorithms, the eigenvalue error decreases as the number of snapshots $N$ increases. Similar to the periodic problem, the eigenvalue errors of Prop-i are less than those of the others for all values of $\sigma^2$ and $N$. Hence, we can conclude that, under the assumption that a sufficiently good initial guess is provided, the proposed model performs much better than the $\ell^2$-optimized DMD model. Meanwhile, we notice from \cref{Fig:hidden_error} that the performance of Prop is similar to that of AK-i, and is a bit better than that of AK in general. This implies that, although the initialization scheme adopted by Prop is not the best one, it still performs as well as AK-i, an ideal case of the $\ell^2$-optimized DMD model. \Cref{Fig:hidden_recon} presents the average of the reconstruction errors $E_{\recon}$ over 1,000 trials for the problem~\cref{hidden_noisy}. Different from Figure~\ref{Fig:hidden_error}, the reconstruction errors of all algorithms are indistinguishable.
Since the eigenvalues $-0.2 \pm 3.7i$ corresponding to the decaying signals barely affect on the reconstruction error, \cref{Fig:hidden_error,Fig:hidden_recon} imply that the eigenvalues corresponding to the growing signals recovered by all algorithms are almost identical while the proposed model yields more accurate values for the eigenvalues corresponding to the decaying signals than the $\ell^2$-optimized DMD model. That is, the proposed model outperforms the state-of-the-art DMD models in view of recovery of hidden dynamics.

\subsection{One-dimensional combustor}
\label{Subsec:Combustor}
Now we demonstrate the proposed DMD model on a realistic physical system, a one-dimensional combustor. Despite its simple configuration, a one-dimensional combustor is ideal for studying the dynamics of a thermoacoustically oscillating system~\cite{guan2019,lee2020,orchini2016,waugh2011}, both experimentally and numerically. In this paper, we generate numerical data representing the pressure oscillation in a one-dimensional combustor under noise. Specifically, the noise that consists of both the additive and multiplicative noise is applied to the system, as per~\cite{waugh2011}.

\begin{figure}[]
\centering
\includegraphics[width=0.45\linewidth]{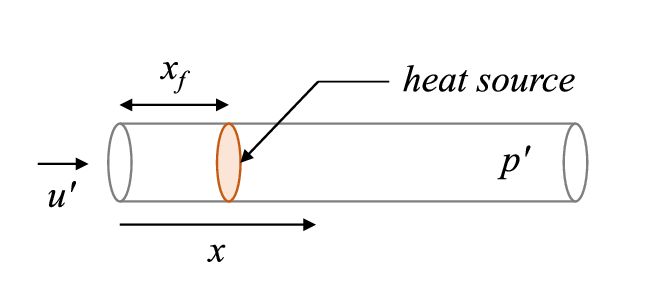} 
\caption{Schematic diagram of the one-dimensional combustor \cite{lee2021}.}
\label{Fig:setup}
\end{figure}

The numerical combustor considered in this system is identical to~\cite{lee2021,waugh2011}, so only a brief description is given in this paper. A schematic of a one-dimensional combustor is shown in \cref{Fig:setup}. In this system, air flows into an open-open cylinder with the velocity fluctuation $u^{\prime}$. A compact heat source is located at $x=x_f$, where $x$ is the distance from the left end of the cylinder. The fluctuation $q^{\prime}$ of the heat release rate from this heat source is given by the following equation~\cite{waugh2011}:
\begin{equation*} \label{eq:noise}
    q^{\prime}=k_Q \left(\sqrt{\left| \frac{1}{3}+u_f'(t-\tau)+d \right| }-\sqrt{\frac{1}{3}}\:\right),
\end{equation*}
where $k_Q$ is the heater power coefficient, $u_f'$ is the velocity fluctuation at the heat source, $\tau$ is the time delay between the flow and the heat release rate. The term $d$ denotes the pink noise acting on the heat source, profiles are shown in subfigures~(b) of \cref{Fig:combustor_recon1,Fig:combustor_recon2,Fig:combustor_recon3}. In this model, the noise perturbs the heat release rate in the mixed form of additive and multiplicative noise. It is worth mentioning that this simple noise model can effectively reproduce the qualitative features of the Ornstein--Uhlenbeck process~\cite{horsthemke1984, waugh2011}.

The momentum and energy equations governing the one-dimensional combustor are as follows:
\begin{align*} 
    \gamma \Ma \pdv{u'}{t} + \pdv{p'}{x} &= 0, \\
    \pdv{p'}{t} + \gamma \Ma \pdv{u'}{x} + \epsilon p' &= q'\delta(x-x_f),
\end{align*}
where $p^{\prime}$ is the pressure fluctuation inside the combustor, $t$ is time, $\gamma$ is the specific heat ratio, $\Ma$ is the Mach number of the mean flow, and $\epsilon$ is the acoustic damping coefficient. In the right-hand side of the second equation, $\delta$ denotes a Dirac delta expressing the local heat release at the heat source. 

By using the Galerkin expansion \cite{lores1973,zinn1971}, a set of ordinary differential equations can be derived from the momentum and energy equations. Specifically, we set $\pdv{u}{x}$=0 and $p'=0$ at both ends of the cylinder, and choose appropriate Galerkin basis functions so that
\begin{subequations}
\begin{align} \label{eq:gal1}
    u' &= \sum_{j=1}^{j_{\max}} \zeta_j \cos{(j \pi x)}, \\
 \label{eq:gal2}
    p' &= -\sum_{j=1}^{j_{\max}} \frac{\gamma \Ma}{j \pi} \dot{\zeta}_j \sin{(j \pi x)},
\end{align}
\end{subequations}
where $j_{\max}$ is the number of superpositioned Galerkin modes and $\zeta_j$ denotes the state variable of the $j$th mode. It should be noted that all Galerkin modes are orthogonal, but are not necessarily the eigenmodes of the system. By replacing $u^{\prime}$ and $p^{\prime}$ of the governing equations with equations \cref{eq:gal1,eq:gal2}, respectively, we obtain
\begin{equation} \label{eq:govfin}
    \ddot{\zeta}_j + (j \pi)^2 \zeta_j + \epsilon_j \dot{\zeta}_j = -k_Q \frac{2 j \pi}{\gamma \Ma} \sin{(j \pi x_f)} \left( \sqrt{\abs{\frac{1}{3}+u_f'(t-\tau)+d}} -\sqrt{\frac{1}{3}}\:\right),
\end{equation}
where $\epsilon_j=0.1 + 0.06\sqrt{j}$ represents the acoustic damping coefficient of the $j$th Galerkin mode~\cite{Gupta2017,lee2021}. \Cref{eq:govfin} is solved numerically using the fourth-order Runge--Kutta method with a time step $\Delta t=0.01$ for $t \in [0, 200]$. The combustor is spatially divided into 500 equal-length segments. We set the parameters of~\cref{eq:govfin} as follows: $\Ma=0.005$, $x_f = 0.25$, $\tau=0.16$, $k_Q=0.0035$, and $\gamma=1.4$, following previous practices \cite{lee_phd, lee2021}. By numerically solving this time-marching problem, we obtain a data matrix $X$ whose rows represent the spatial distribution of $p^{\prime}$ and columns represent the time.

\begin{figure}[]
\centering
\includegraphics[width=0.9\linewidth]{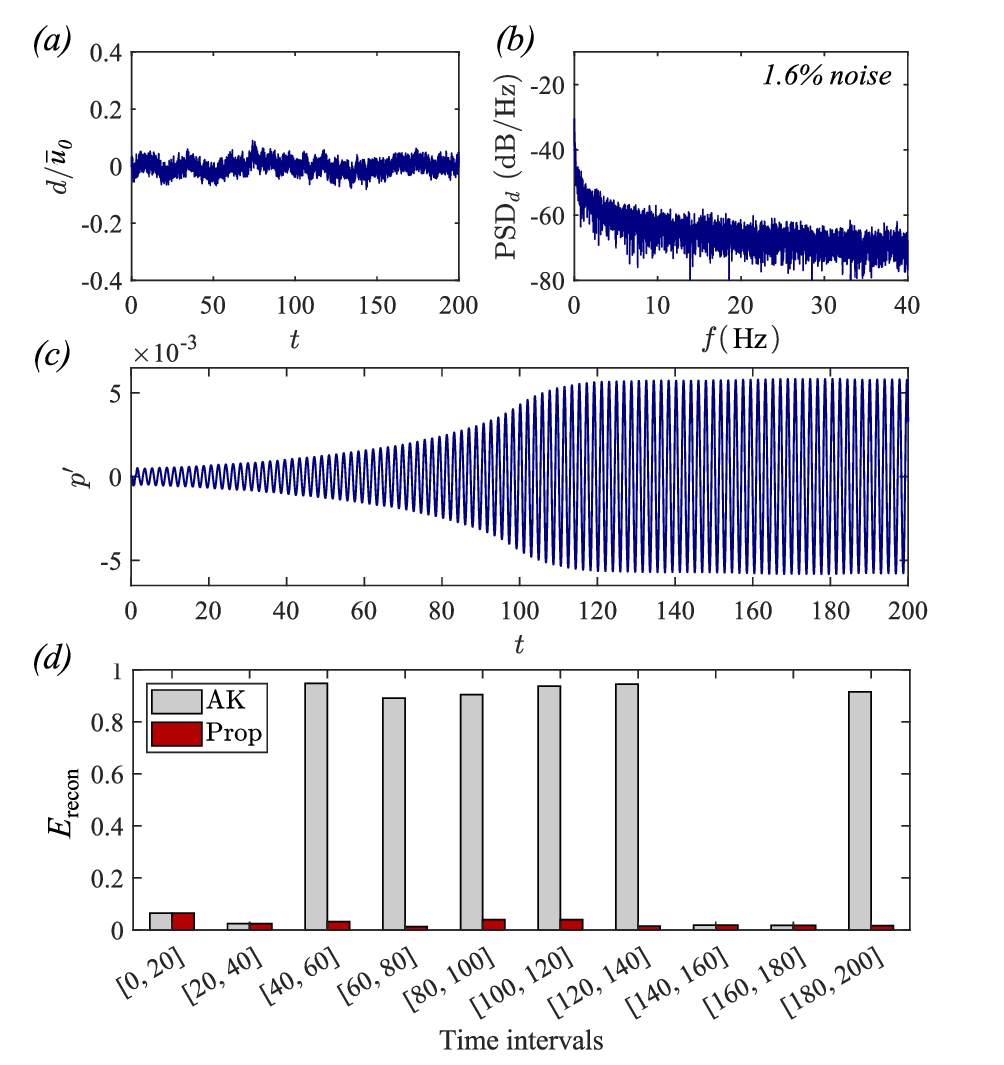} 
\caption{(a, b) Noise profile, (c) pressure signal and (d) reconstruction error for the one-dimensional combustor problem under weak noise.}
\label{Fig:combustor_recon1}
\end{figure}

\begin{figure}[]
\centering
\includegraphics[width=0.9\linewidth]{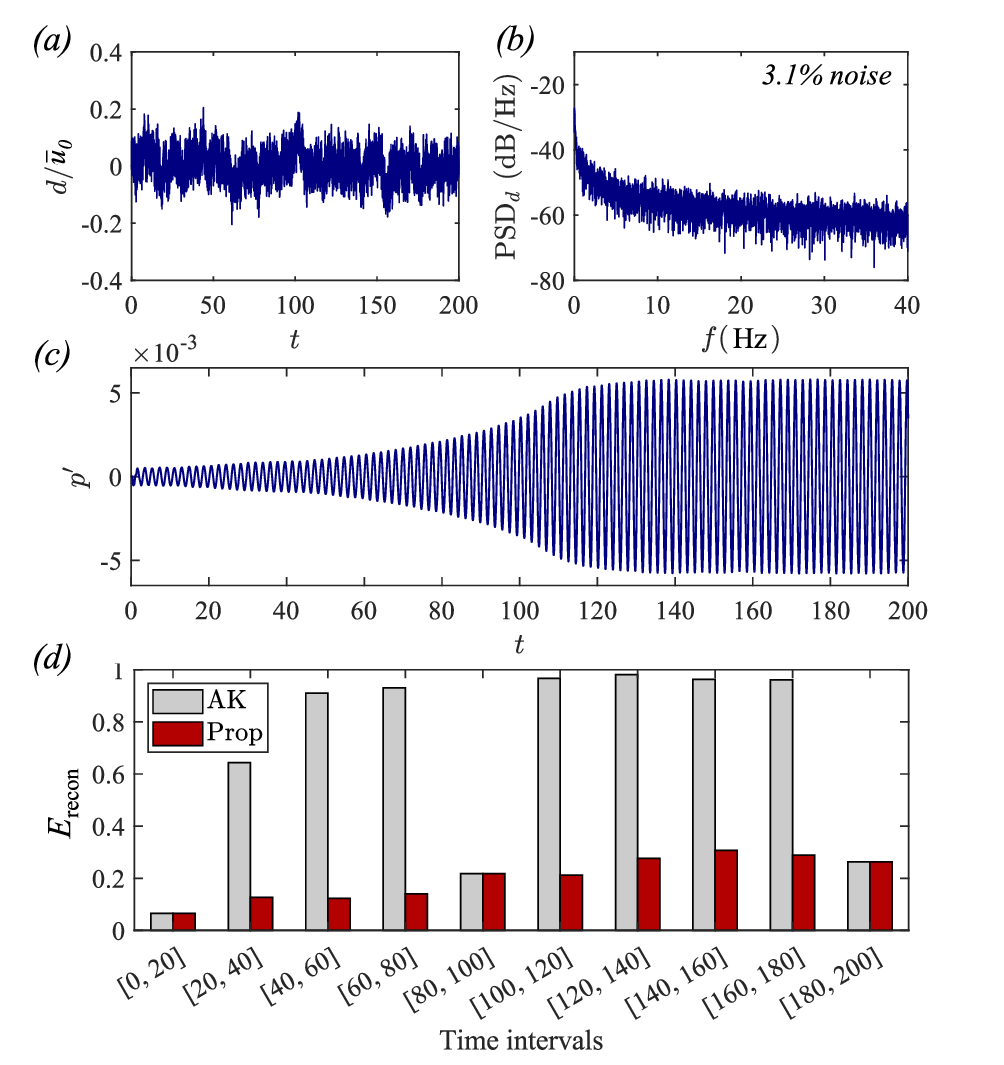} 
\caption{(a, b) Noise profile, (c) pressure signal and (d) reconstruction error for the one-dimensional combustor problem under intermediate noise.}
\label{Fig:combustor_recon2}
\end{figure}

\begin{figure}[]
\centering
\includegraphics[width=0.9\linewidth]{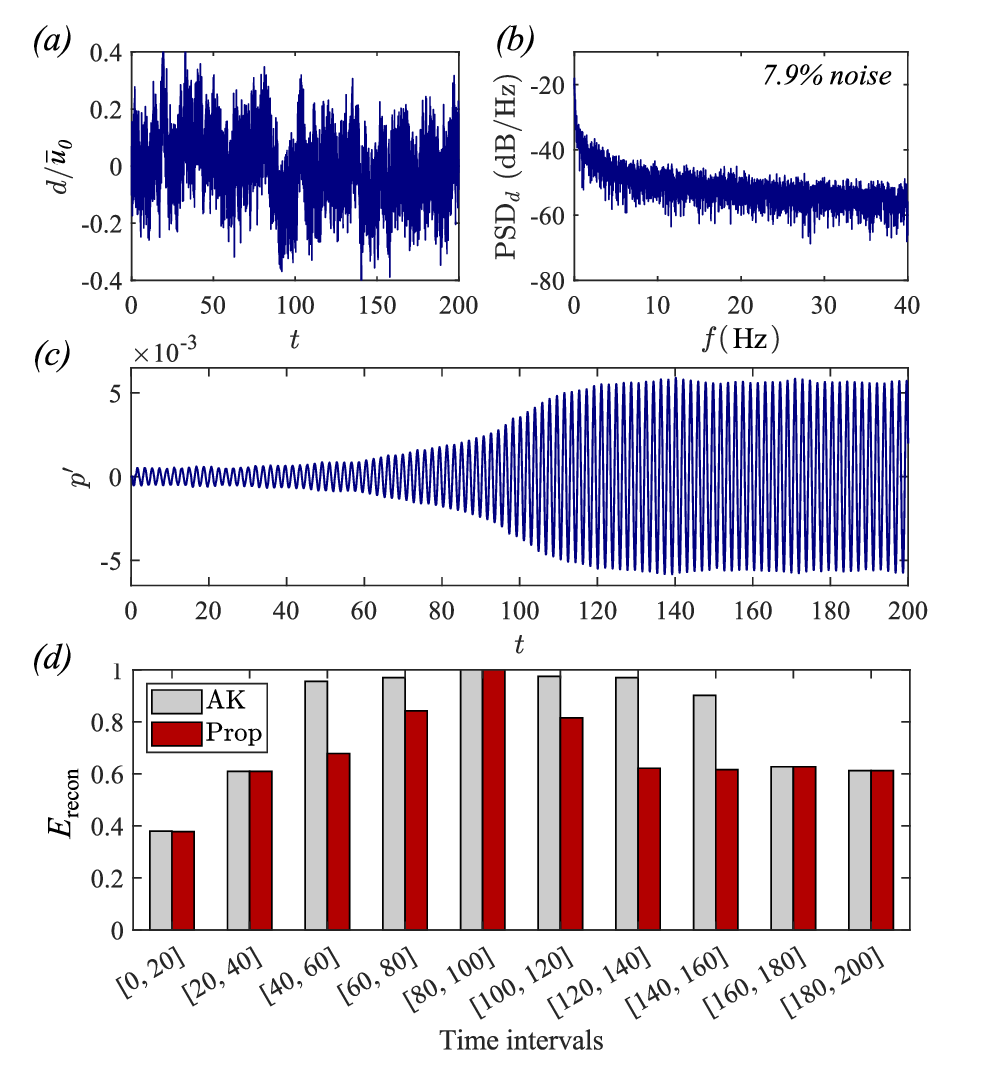} 
\caption{(a, b) Noise profile, (c) pressure signal and (d) reconstruction error for the one-dimensional combustor problem under strong noise.}
\label{Fig:combustor_recon3}
\end{figure}

We normalize the noise intensity with the average flow fluctuation amplitude~($\bar{d}/\bar{u}_0$) and consider three distinctive cases: weak~($\bar{d}/\bar{u}_0$=0.0016), intermediate~($\bar{d}/\bar{u}_0$=0.0031) and strong~($\bar{d}/\bar{u}_0$=0.0079) noise. The profiles of weak, intermediate and strong noise are shown in subfigures (a, b) of \cref{Fig:combustor_recon1,Fig:combustor_recon2,Fig:combustor_recon3}, respectively. 

The pressure fluctuation signals at the pressure antinode~($x=0.5$) are shown in subfigures (c) of \cref{Fig:combustor_recon1,Fig:combustor_recon2,Fig:combustor_recon3}. Regardless of the noise intensity, the pressure fluctuation develops gradually until $t\approx110$ when the nonlinearity starts to dominate the dynamics of the system. In order to capture the local linearity of the system, we divide the pressure signal into ten time sections $[0, 20]$, $[20, 40]$, \dots, $[180, 200]$, and apply the proposed DMD model with $\eta=10^3$ at each section. Specifically, we use $R=10$~(i.e., five pairs of eigenvectors) to decompose the noisy signal segments. We then reconstruct the pressure signal using the obtained modes and compare it with the clean (zero-noise) data.
The reconstruction error $E_{\recon}$ defined in~\cref{recon} is measured at each segment.

It can be found from \cref{Fig:combustor_recon1}(d) that the proposed DMD model can accurately decompose the signal with minimal reconstruction error. Specifically, $E_{\recon}$ is 6.5\% in the first segment and is less than 4\% in all other segments. A significant decrease in the reconstruction error is observed when compared to the $\ell^2$-optimized DMD model. Considering that the noise acting on the system has both additive and multiplicative natures, this result shows the robustness of the proposed DMD model to the multiplicative noise.

In the intermediate noise case~(see \cref{Fig:combustor_recon2}), the reconstruction error increases, showing $E_{\recon}$ values between $6.5\%$ and $31\%$. DMD results are comparatively accurate in the first half of the signal where linear growth is observed, but the inaccuracy increases in the second half, where the strong nonlinearity comes into play. Finally, when the noise intensity is further increased~(see \cref{Fig:combustor_recon3}), the reconstruction error becomes very high. This implies that the DMD method, which assumes linear temporal development, is difficult to be used for reconstruction anymore. Nevertheless, it is worth mentioning that the reconstruction error is smaller in the proposed model, compared to the $\ell^2$-optimized DMD model.

At this point, it is of interest to find the cause of such a decrease in reconstruction error. Because both the $\ell^2$-optimized DMD model and the proposed model heavily depend on the initial $\alpha$, one may hypothesize that an ideal initial guess may lead to equally good results in both models. In order to check such a claim, we test both models using the initial $\alpha$  values obtained from the clean data~(i.e., zero-noise combustor data). The result shown in \cref{Fig:combustor_recon4} reveals that an ideal initial value significantly reduces the reconstruction error of the $\ell^2$-optimized DMD model, while its effect on the proposed method is minimal~(see \cref{Fig:combustor_recon3} for comparison). This result indicates that the proposed method by itself can find a nice initial value. It is also notable that, at some points~(e.g., time intervals $[60, 80]$ and $[120, 140]$ of \cref{Fig:combustor_recon4}), the advantage of the proposed method over the existing method is substantiated regardless of the initial value.

\begin{figure}[]
\centering
\includegraphics[width=0.9\linewidth]{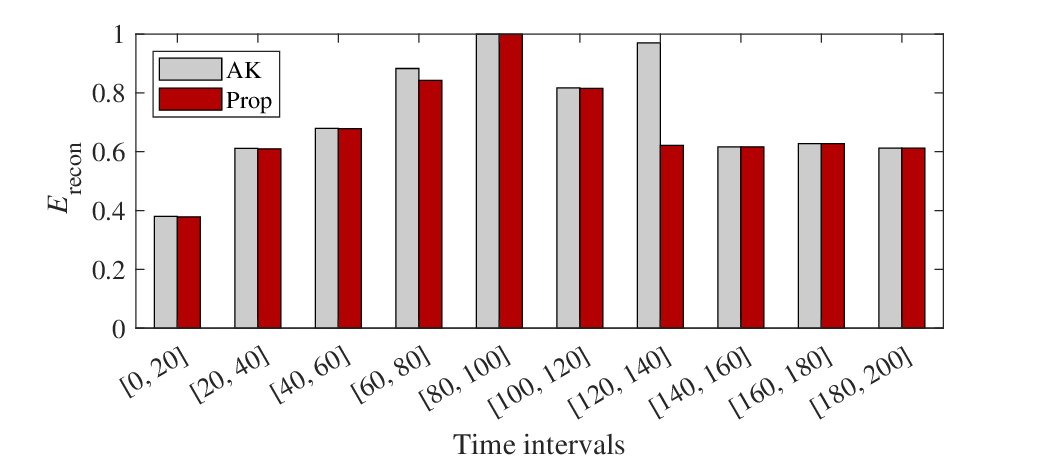} 
\caption{Reconstruction error for the one-dimensional combustor problem under strong ($\bar{d}/\bar{u}_0$=0.0079) noise, using initial $\alpha$ obtained from clean~(zero noise) data.}
\label{Fig:combustor_recon4}
\end{figure}

\begin{figure}[]
\centering
\includegraphics[width=0.7\linewidth]{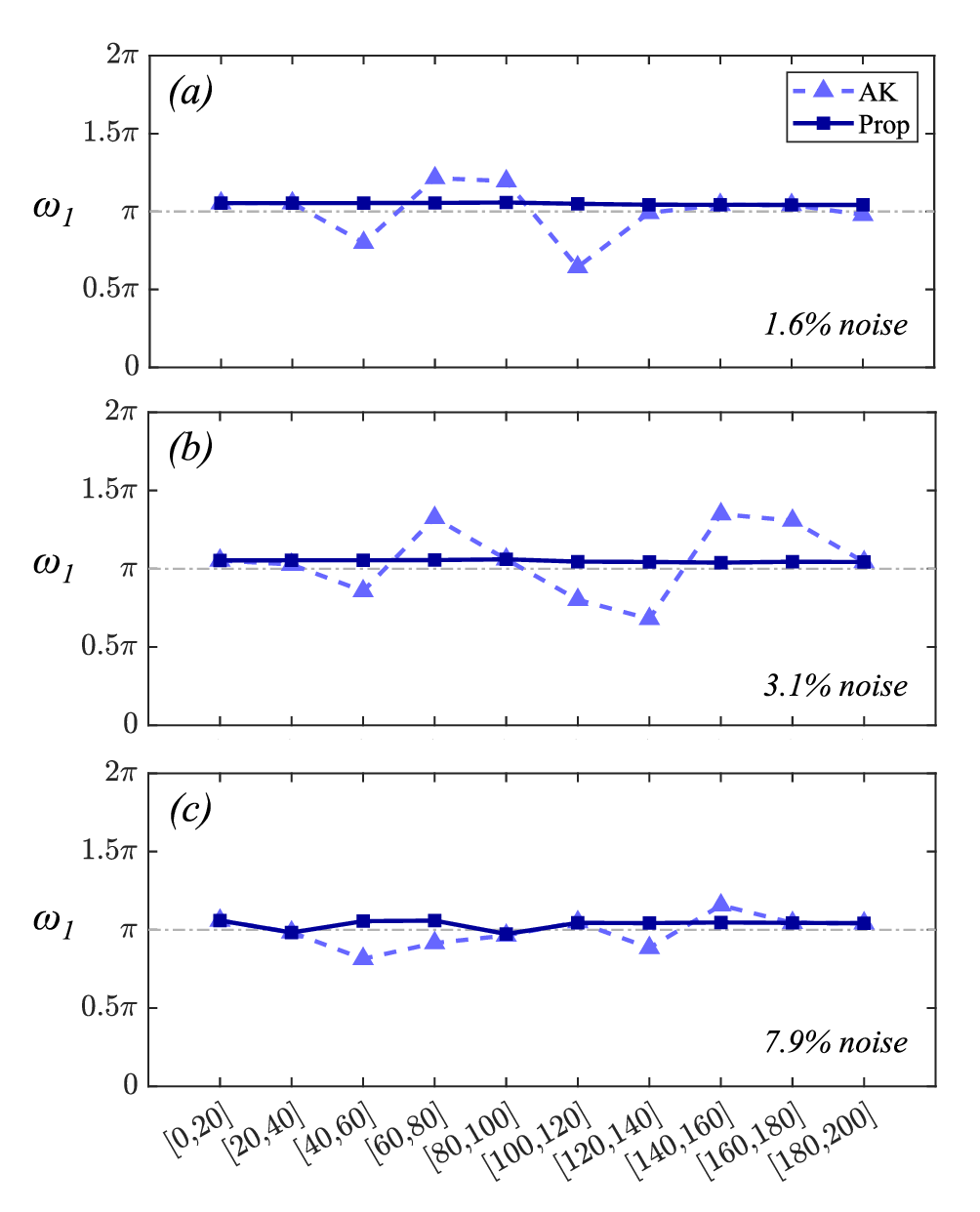} 
\caption{Primary angular frequencies of the combustor pressure oscillation identified from the $\ell^2$-optimized DMD model and the proposed model at (a) 1.6\% noise, (b) 3.1\% noise and (c) 7.9\% noise. Gray lines indicate the analytical primary angular frequency $\pi$. Horizontal axes are identical to the time intervals displayed in \cref{Fig:combustor_recon1,Fig:combustor_recon2,Fig:combustor_recon3}.}
\label{Fig:combustor_freq}
\end{figure}

In the one-dimensional combustor used in our numerical experiment, the mode of oscillation is determined by the duct mode of an open-open tube. Specifically, the angular frequency of the $j$th duct mode is equal to $j\pi$. It is known from the previous study that the primary mode of oscillation carries the majority of the energy~\cite{lee_phd}. We therefore assess whether the proposed model can properly identify the primary oscillation frequency. \cref{Fig:combustor_freq} shows the primary angular frequency~($\omega_1$) obtained from the $\ell^2$-optimized DMD and proposed models under different noise intensities. It is notable that, while both models can reasonably identify the primary mode at $\omega_1=\pi$, the proposed model better captures the analytical result even under a high level of noise. This result highlights the importance of the denoising model specific to multiplicative noise.

\section{Conclusion}
\label{Sec:Conclusion}
In this study, we proposed a novel optimized DMD model that is robust to multiplicative noise. Combining the ideas of the $\ell^2$-optimized DMD model~\cite{AK:2018} and the Aubert--Aujol denoising model~\cite{AA:2008}, we developed a framework that can accurately decompose a dynamical system under the multiplicative noise.
We applied the framework to three numerical examples, including a realistic physical system, and showed that the accuracy of the proposed DMD model had been improved compared to other DMD techniques. This study highlights that designing DMD models tailored for the type of noise can improve the accuracy of reconstruction.

This work suggests several interesting topics for future research. For instance, it is well-known that the performance of nonconvex variational models highly depends on the choice of an initial guess~\cite{YL:2020}. Indeed, numerical results in \cref{Sec:Numerical} showed that a better initial guess for the proposed model yields a better reconstruction result. It is therefore an important task to design a good initialization scheme for the proposed model. Realization of the optimal performance in the proposed model, which requires the design of an appropriate initialization scheme, remains as future work.
 
Nevertheless, it is encouraging that the proposed DMD model showed an outstanding reconstruction performance when applied to a realistic physical system, namely the one-dimensional combustor. This implies that the proposed optimized DMD model can contribute to the accurate decomposition of various practical systems in nature and engineering, especially those under the influence of multiplicative noise \cite{brand1985, fox1984, granwehr2007, short1982}.

\section*{Acknowledgement}
This work was inspired by the discussion with Professor Chae Hoon Sohn at Sejong University regarding the development of robust model reduction techniques for combustion dynamics. The authors would like to thank him for his insightful comments and assistance in the early stage of this work.

\bibliographystyle{siamplain}
\bibliography{refs_ODMD}

\end{document}